\documentclass[12pt,reqno]{amsart}

\usepackage{lineno}

\usepackage{graphicx}
\usepackage{cases}
\usepackage{amscd}
\usepackage{amsfonts}
\usepackage{amsmath}
\usepackage{graphicx}
\usepackage{subfigure}
\usepackage{amsthm}
\usepackage[
pdfauthor={derajan},
pdftitle={How to do this},
pdfstartview=XYZ,
bookmarks=true,
colorlinks=true,
linkcolor=blue,
urlcolor=blue,
citecolor=blue,
pdftex,
bookmarks=true,
linktocpage=true, 
hyperindex=true
]{hyperref}
\usepackage{hyperref}
\theoremstyle{remark}
\newtheorem{example}{\textbf{Example}}[section]
\numberwithin{equation}{section}
\usepackage{color}
\usepackage{datetime}

\allowdisplaybreaks
\makeatletter
\newcommand\figcaption{\def\@captype{figure}\caption}
\newcommand\tabcaption{\def\@captype{table}\caption}
\makeatother
\usepackage{geometry}
\usepackage{algorithm}
\usepackage{multirow}

\def\bq{\begin{equation}}
\def\eq{\end{equation}}
\def\br{\begin{eqnarray}}
\def\er{\end{eqnarray}}
\def\brr{\bq\begin{array}{rlll}}
\def\err{\end{array}\eq}

\def\text#1{\hbox{#1}}
\newtheorem{thm}{Theorem}[section]

\newtheorem{rem}{Remark}[section]

\newcommand{\bsub}{\begin{subequations}}
\newcommand{\esub}{\end{subequations}$\!$}

\title[ Positivity-preserving schemes for reduced PNP system]{Unconditional positivity-preserving and energy stable schemes for a reduced Poisson-Nernst-Planck system}
\author{  Hailiang Liu and Wumaier Maimaitiyiming}
\address{Iowa State University, Mathematics Department, Ames, IA 50011}
\email{hliu@iastate.edu}\email{wumaierm@iastate.edu} 
\begin{document}

\begin{abstract} 
The Poisson-Nernst-Planck (PNP) system is a widely accepted model for simulation of ionic channels. In this paper, we design, analyze,
and numerically validate  a second order unconditional positivity-preserving scheme for solving a reduced PNP system, which can well approximate the three dimensional ion channel problem.  Positivity of numerical solutions is proven to hold true independent of the size of time steps and the choice of the Poisson solver.   The scheme is easy to implement without resorting to any iteration method. Several numerical examples further confirm the positivity-preserving property, and demonstrate the accuracy, efficiency,  and robustness of the proposed scheme, as  well as the fast approach to steady states.
\end{abstract}

\maketitle

\section{Introduction} 
Biological cells exchange chemicals and electric charge with their environments through ionic channels in the cell membrane walls. Examples include signaling in the nervous system and coordination of muscle contraction,  see \cite{Cu09} for a comprehensive introduction.  Mathematically the flow of ions can be modeled by drift-diffusion equations such as the Poisson-Nernst-Planck (PNP) system, see e.g. \cite{Coalson05,Eis96,Eis07,Eis04}.  

In this investigation we design, analyze and numerically validate positivity-preserving algorithms to solve time-dependent drift-diffusion equations. 
As a first step, in this paper we focus on a reduced model derived by Gardner et al \cite{Eis04} as an approximation to the full  three dimensional (3D) 
PNP system.  Let us first recall the full model and its reduction. 

\subsection{Mathematical models}  The  general setup in \cite{Eis04} is a flow of positive and negative ions in water in a channel plus surrounding baths in an electric field against a background of charged atoms on the channel protein. The distribution of charges is described by continuum particle densities $c_i(\mathbf{x},t)$ for the mobile ions (such as $K^+, N_a^+, C_a^{++},\cdots $). The flow of ions can be modeled by the PNP system of $m$ equations 
\begin{equation}\label{PNP0}
\begin{aligned}
   \partial_t c_i&= -\nabla\cdot J_i,\quad i=1, \cdots, m; \; \mathbf{x}\in \Omega \subset \mathbb{R}^3, \; t>0, \\
   J_i & = -(D_i\nabla c_i+z_i\mu_ic_i\nabla \psi),\\
-\nabla \cdot(\epsilon \nabla \psi)&=\sum_{i=1}^{m}q_ic_i-e\rho,
\end{aligned}
\end{equation}
where $J_i$ is the flux density, in which $D_i$ is the diffusion coefficient, $\mu_i$ the mobility coefficient which is related to the diffusion coefficient via Einstein's relation $\mu_i=\frac{D_i}{k_BT_0}$, where $k_B$ is the Boltzmann constant and $T_0$ is the absolute temperature \cite{Cu09}. 
In the Poisson equation,  $\epsilon$ is the dielectric coefficient, $q_i$ the ionic charge for each ion species $i$, $\rho=\rho(\mathbf{x})$ the permanent fixed charge density, and $e$ the proton charge. The coupling parameter $z_i = q_i/e$.  In general, the physical parameters $\epsilon$, $\mu_i$  and $D_i$ are functions of $\mathbf{x}$.  Let us mention that the case of no permanent charge does not pertain to biological channels. Even channels without permanent charge (in the form of so called acid and base side chains) have large amounts of fixed charge in their (for example) carbonyl bonds( see, e.g., \cite{WS15} and references therein).

 The derivation of the Nernst-Planck equation typically follows two steps, namely, using the energy variation to obtain the chemical potential and then using Fick's laws of 
 diffusion to attain the Nernst-Planck equation (see e.g. \cite{FB97}).  In the charge dynamics modeled by the traditional NP equation, mobile ions are treated as volume-less 
 point charges. In order to incorporate more complex effects such as short-range steric effect and long range Coulomb correlation, modifications of the PNP equations 
 were derived ( see, e.g., \cite{PXZ18} and references therein). Nonetheless, the scheme methodology proposed in this paper 
 can well be adapted to solve such modified PNP systems. 

The 3D geometry of the ion channel can be approximated by a reduced problem along the axial direction $x$, with a cross-sectional area $A(x)$ \cite{SN09, SGNE08}.  Subject to a further rescaling as in  \cite{GLE2018}, the  corresponding PNP system (\ref{PNP0}) reduces to 
the following equations 
 \begin{equation}\label{PNP2}
\begin{aligned}
  \partial_t c_i&=\frac{1}{A(x)} \partial_x (A(x)D_i(\partial_x c_i+z_ic_i\partial_x \psi)),  \quad   x \in \Omega=[0,\ 1],  \quad  t>0,\\
 -\frac{1}{A(x)} \partial_x (\epsilon A(x) \partial_x \psi)&=\sum_{i=1}^{m}z_ic_i-\rho(x),  \quad    x \in \Omega,  \quad  t>0.
   \end{aligned}
\end{equation}
For ionic channels, an important characteristic is the so-called current-voltage relation, which can characterize permeation and selectivity properties of ionic channels (see \cite{WS-IV-08} and references therein). For  (\ref{PNP0}), the electric current density (charge flux) is $J=\sum_{i=1}^m q_i J_i$.
Such quantity for (\ref{PNP2}) reduces to
 \begin{equation}\label{J}
J=-\sum_{i=1}^m z_i D_iA(x) ( \partial_x c_i + z_ic_i \partial_x \psi).
\end{equation}
System (\ref{PNP2})  is a parabolic/elliptic system of partial differential equations, boundary conditions for both $c_i$ and $\psi$ can be Dirichlet or Neumann.  

In order to solve the above reduced system,  we consider initial data 
$$
c_i(x, 0)=c_i^{\rm in}(x) \geq 0, \quad x\in \Omega. 
$$

\subsection{Boundary conditions and model properties}  We consider two types of boundary conditions.   The first is the Dirichlet boundary condition, 
\begin{equation}\label{Di}
 c_i(0,t)=c_{i,l}, \quad c_i(1,t)=c_{i,r}; \quad \psi(0,t)=0, \psi(1,t)=V,\quad   t>0,
\end{equation}
where $c_{i,l}, c_{i, r}$ are non-negative constants, and $V$ is a given constant. This is the setting adopted in \cite{Eis04}.  One important solution property is 
\begin{equation}\label{po}
c_i(x, t) \geq 0, \quad x\in \Omega, \; t>0.
\end{equation}
Another set of boundary conditions is as follows: 
\begin{equation}\label{Zf}
\begin{aligned}
&\partial_x c_i+z_ic_i\partial_x \psi=0,\quad x=0,1, \quad t>0,\\
& (- \eta \partial_x \psi+\psi)|_{x=0}=\psi_{_-},\quad ( \eta \partial_x \psi+\psi)|_{x=1}=\psi_{_+}, \quad t>0,
\end{aligned}
\end{equation}
where $\psi_{_-}, \psi_{_+}$ are given constants, the size of parameter $\eta $ depends on the properties of the surrounding membrane \cite{Eis13}.  Here the first one is the zero-flux boundary condition for the transport equation, and the second is the Robin boundary condition for the Poisson equation.  Such boundary condition is adopted in \cite{Eis13} to model the effects of partially removing the potential from the ends of the channel. For system (\ref{PNP2}) with this boundary condition,  solutions have non-negativity, mass conservation,  and free energy dissipation properties, i.e.,  (\ref{po}),  
\begin{equation}\label{mass}
\int_{\Omega}A(x)c_i(x,t)dx=\int_{\Omega}A(x)c_i^{in}(x)dx, \quad  t>0,\quad i=1,\cdots,m, \; \text{and} 
\end{equation}

\begin{equation}\label{energydiss1}
\frac{dE}{dt}=-\int _{\Omega} \sum_{i=1}^m A(x)D_ic_i|\partial _{x} ( \log c_i+z_i\psi)|^2 dx  \leq 0        ,
\end{equation}
where the total energy $E$ associated to (\ref{PNP2}) is defined (see \cite{Eis13}) by
\begin{equation}\label{energy1}
E=\int_{\Omega}A(x) \bigg(\sum_{i=1}^m c_i\log c_i+\frac{1}{2}(\sum_{i=1}^mz_ic_i-\rho)\psi\bigg) d x+\frac{\epsilon}{2 \eta }(\psi_{_+}A(1)\psi(1)+\psi_{_-}A(0)\psi(0)).
\end{equation}
The positivity-preserving property is of special importance, since negative values in density  would violate the physical meaning of the solution and may destroy the energy dissipation law (\ref{energydiss1}). Numerical techniques addressing the positivity preserving property have been introduced in various application problems, see e.g.
 \cite{JSY19,LYY18}. In this paper, we construct second order accurate unconditional positivity-preserving schemes for solving (\ref{PNP2}) subject to  two types of boundary conditions. For the zero-flux boundary condition, the schemes will be shown to satisfy mass conservation and a discrete energy dissipation law. 

\subsection{Related works} 
Numerical methods for solving the PNP system of equations have been studied extensively;  see e.g., \cite{Eis04,HP16,HEL11,LHMZ10, ZCW11}. We also refer to \cite{Gwei16} for a review on the PNP model and its generalizations for ion channel charge transport.

For the reduced PNP system (\ref{PNP2}),
the finite difference scheme with TR--BDF2 time integration was first pursued in  \cite{Eis04} to simulate an ionic channel. For the one dimensional PNP system, the second order implicit finite difference scheme proposed in \cite{Eis13} can preserve total concentration of ions with the aid of a special boundary discretization,  but numerical solutions may not be positive or energy dissipating. An improved scheme, further introduced in \cite{Flavell7}, can preserve a discrete form of energy dissipation law up to $O(\tau^2+h^2)$, where $\tau$ is the time step, and $h$ is the spatial mesh size. In \cite{CH19} the authors proposed an adaptive conservative finite volume method on a moving mesh that maintains solution positivity. The second oder finite difference scheme in \cite{LW14} is explicit and shown to preserve positivity, mass conservation,  and energy dissipation, while the positivity-preserving property is ensured if $\tau=O(h^2)$. Further extension in \cite{LW17} is a free energy satisfying discontinuous Galerkin scheme of any high order, where positivity-preserving property is realized by limiting techniques. The finite element scheme obtained by the method of lines approach in \cite{Me16} preserves positivity of the solutions and a discrete energy dissipation law. Recently in \cite{Hjingwei19} the authors presented a fully implicit finite difference method where both positivity and energy decay are preserved. In their scheme a fixed point iteration is needed for solving the resulting nonlinear system.  These schemes are either explicit or fully implicit in time, the former require a time step restriction for preserving the desired properties  while the later  preserve desired properties unconditionally but they had to be solved by some iterative solvers.

In this paper we design schemes to preserve all three desired properties of solution: positivity, mass conservation, and energy dissipation, by following  \cite{LiuWu2018}, in which a second order finite-volume method was constructed for the class of nonlinear nonlocal equations
\begin{equation}\label{LiuWu}
\partial_t c =\nabla \cdot (\nabla c+c \nabla (V(\mathbf{x})+W*c )).
\end{equation}
The key ingredients include a reformulation  of  the equation in its non-logarithmic Landau form and the use of the implicit-explicit time discretization, these together ensure the positivity-preserving property without any restriction on the size of time steps (unconditional!) and do not require iterative solvers. 

\subsection{Contributions and organization of the paper} 
Our scheme construction is based on the reformulation
\begin{equation}\label{RF}
   A(x)\partial_t c_i(x,t) =\partial_x (A(x)D_ie^{-z_i\psi(x,t)}\partial_x( c_i(x,t)e^{z_i\psi(x,t)})),
\end{equation}
of the transport equation in (\ref{PNP2}). Similar formulation has been used in \cite{LiuWu2018} and in earlier works \cite{LY12, LW14}. Following \cite{LiuWu2018},  we adopt a semi-implicit time discretization of (\ref{RF}):
\begin{equation}\label{time}
A(x)\frac{c_i^{n+1}(x)-c_i^{n}(x)}{\tau}=\partial_x \left(    A(x)D_i e^{-z_i \psi^n(x)} \partial_x (c_i^{n+1}(x) e^{z_i \psi^n(x)}) \right).
\end{equation}
The feature of such discretization is that it is a linear equation in $c_i^{n+1}(x)$, and easy to solve numerically. For spatial discretization, we use the central finite volume approach. The coefficient matrix of the resulting linear system is an M-matrix and right hand side is a nonnegative vector, thus positivity of the solution is ensured without any time step restriction.

The main contribution in this paper includes the model reformulation, proofs of unconditional positivity-preserving properties for two types of  boundary conditions,  and of mass conservation and energy dissipation properties for zero flux boundary conditions (\ref{Zf}). In addition, the positivity-preserving property is shown to be independent of the choice of Poisson solvers. Our implicit-explicit scheme is easy to implement and efficient in computing numerical solutions over long time.

The paper is organized as follows. In section 2, we derive our numerical scheme for a model equation. Theoretical analysis of unconditional positivity is provided. In section 3, we formulate our scheme to the PNP system and prove positivity, mass conservation and energy dissipation properties of the scheme. Numerical examples are presented in section 4. Finally, concluding remarks are given in section 5.

\section{Numerical methods for a model equation}
In this section,  we first demonstrate the key ideas through a model problem. 
Let $u(x, t)$ be an unknown density, satisfying  
\begin{equation}\label{Model1}
\begin{aligned}
   A(x)\partial_t u(x,t)&= \partial_x (B(x)(\partial_xu(x,t)-u(x,t)\partial_x \phi (x,t))  ),  \quad   x \in \Omega=[0,\ 1],  \quad  t>0,\\
  u(x,0)&=u^{in}(x), \quad x \in \Omega,  \\
\end{aligned}
\end{equation}
where $A(x)>0, B(x)>0$  are given functions, and $\phi(x,t)$ is either known or can be obtained from solving another coupled equation.  For this model problem, we consider two types of boundary conditions: \\ 
(i) the Dirichlet boundary condition
\begin{equation}\label{Di2}
    u(0,t)=u_l, \quad u(1,t)=u_r, \quad    t > 0,
\end{equation}
and (ii) the zero flux boundary condition
\begin{equation}\label{Zf2}
\partial_xu(x,t)-u(x,t)\partial_x \phi (x,t) =0, \quad x=0,1, \quad t>0.
\end{equation} 

\subsection{Scheme formulation} 
Let $N$ be an integer, and the domain $\Omega=[0, \ 1]$ be partitioned into computational cells  $I_j=[x_{j-1/2}, \ \ x_{j+1/2}]$ with cell center $x_{j}=x_{j-1/2}+\frac{1}{2}h$, for $j\in \{1,2, \cdots, N\},$   $x_{1/2}=0$ and $x_{N+1/2}=1.$ For simplicity, uniform mesh size $h=\frac{1}{N}$ is adopted. Discretize $t$ uniformly as $t_n=\tau n$, where $\tau$ is time step.

From the reformulation
\begin{equation}\label{Model2}
   A(x)\partial_t u(x,t)= \partial_x (B(x)e^{\phi(x,t)}\partial_x( u(x,t)e^{-\phi(x,t)}))
\end{equation}
of (\ref{Model1}), we consider a semi-implicit time discretization as follows: 
\begin{equation}\label{time1}
A(x)\frac{u^{n+1}(x)-u^{n}(x)}{\tau}=\partial_x \left(    B(x)e^{\phi^n(x)} \partial_x (u^{n+1}(x) e^{-\phi^n(x)}) \right),
\end{equation}
where $u^{n}(x)\approx u(x,t_n),$ $\phi^n(x)\approx \phi(x,t_n).$ 
 Let $u^n_j\approx \frac{1}{h}\int_{I_j}u^n(x)dx$, and $ A_{j}=\frac{1}{h}\int_{I_j} A(x)dx$,
 then a fully-discrete scheme of (\ref{time1}) can be given by 
 \begin{equation}\label{fully}
A_j\frac{u^{n+1}_{j}-u^n_{j}}{\tau}=\frac{U_{j+1/2}-U_{j-1/2}}{h},
\end{equation}
where the flux on interior interfaces are defined by
\begin{equation}\label{UU}
U_{j+1/2}=B_{j+1/2} e^{\phi^n_{j+1/2}}\frac{u^{n+1}_{j+1}e^{-\phi^n_{j+1}}-   u^{n+1}_{j}e^{-\phi^n_{j}}   }{h}, \quad j=1,2,\cdots, N-1.
\end{equation}
Here $B_{j+1/2}=B(x_{j+1/2})$; For $\phi^n_{j+1/2}$ we either use $\phi(x_{j+1/2}, t_n)$ if $\phi(x, t)$ is given, or 
$$
\phi^n_{j+1/2}=\frac{\phi_j^n+ \phi^n_{j+1}}{2},
$$
where  $\phi^n_{j}$ is a numerical approximation of $\phi(x_j, t_n)$. 

The boundary fluxes are given as follows: \\
(i) for the Dirichlet boundary condition (\ref{Di2})
 \begin{equation}\label{U11}
\begin{aligned}
 U_{1/2}&=B_{1/2}e^{\phi^n_{1/2}}\frac{2(u^{n+1}_{1}e^{-\phi^n_1}-u_{l}e^{-\phi^n_{1/2}})}{h}, \\
 U_{N+1/2}&=B_{N+1/2}e^{\phi^n_{N+1/2}}\frac{2(u_{r}e^{-\phi^n_{N+1/2}}-u^{n+1}_{N}e^{-\phi^n_{N}})}{h};
 \end{aligned}
\end{equation}
(ii) for the zero flux boundary condition (\ref{Zf2}), 
\begin{equation}\label{Zf3}
U_{1/2}=U_{N+1/2}=0.
\end{equation}
In either case, the initial data are determined by
$$
u_{j}^0=\frac{1}{h}\int_{I_j}u^{in}(x)dx, \quad j=1,\cdots, N.
$$
Before turning to the analysis of solution properties, we comment on these boundary fluxes.
\begin{rem} The factor 2 in the boundary flux (\ref{U11}) suffices to ensure the first order accuracy in the approximation of 
$$
B(x)e^{\phi(x,t)}\partial_x( u(x,t)e^{-\phi(x,t)})
$$
at the boundary;  see \cite{Eymard}.  However, the following flux without the factor 2, i.e. 
 \begin{equation}\label{U0}
\begin{aligned}
 U_{1/2}&=B_{1/2}e^{\phi^n_{1/2}}\frac{(u^{n+1}_{1}e^{-\phi^n_1}-u_{l}e^{-\phi^n_{1/2}})}{h}, \\
 U_{N+1/2}&=B_{N+1/2}e^{\phi^n_{N+1/2}}\frac{(u_{r}e^{-\phi^n_{N+1/2}}-u^{n+1}_{N}e^{-\phi^n_{N}})}{h},
 \end{aligned}
\end{equation}
can produce only a zeroth order approximation at the boundary.  Order loss of accuracy has been observed in our numerical tests
when (\ref{U0}) is used.   

An alternative  boundary flux for (i) is a second order approximation of the form 
 \begin{equation}\label{U22}
\begin{aligned}
U_{1/2}&=B_{1/2}e^{\phi^n_{1/2}}\frac{-\frac{1}{3}u^{n+1}_{2}e^{-\phi^n_{2}}+3u^{n+1}_{1}e^{-\phi^n_{1}}-\frac{8}{3}u_le^{-\phi^n_{1/2}}}{h}, \\
U_{N+1/2}&=B_{N+1/2}e^{\phi^n_{N+1/2}}\frac{\frac{1}{3}u^{n+1}_{N-1}e^{-\phi^n_{N-1}}-3u^{n+1}_{N}e^{-\phi^n_{N}}+\frac{8}{3}u_re^{-\phi^n_{N+1/2}}}{h}.
 \end{aligned}
\end{equation} 
However, it is known that the first order boundary flux  does not destroy the second order accuracy of the scheme, we refer to \cite{LW18} for a such result  regarding the Shortley-Weller method. Hence throughout the paper, we will not discuss high order boundary fluxes such as (\ref{U22}). 
\end{rem} 

\subsection{Positivity} 
It turns out that both schemes,  (\ref{fully})-(\ref{UU})-(\ref{U11}) and  (\ref{fully})-(\ref{UU})-(\ref{Zf3}),  preserve positivity of numerical solutions 
without any time step restriction.
\begin{thm} Scheme (\ref{fully})-(\ref{UU}) with either 
(i) (\ref{U11}) and $u_l\geq 0, u_r\geq 0$,  or 
(ii) (\ref{Zf3}),  
is positivity-preserving, in the sense that 
if $u^n_{j}\geq 0$ for all $j=1,\cdots, N$,    
then 
$$
u^{n+1}_{j}\geq 0 \; \text{for all}\;  j=1,\cdots, N.
$$
\end{thm}
\begin{proof} 
Set mesh ratio $\lambda=\frac{\tau}{h^2}$ and introduce $G_{j}=u^{n+1}_j e^{-\phi^n_{j}}$, so that \\
(i) scheme  
(\ref{fully}), (\ref{UU}) and (\ref{U11}) can be rewritten as 
\begin{equation}\label{G}
\begin{aligned}
&(A_1e^{\phi^n_1}+\lambda B_{3/2}e^{\phi^n_{3/2}}+2\lambda B_{1/2}e^{\phi^n_{1/2}})G_1-\lambda B_{3/2}e^{\phi^n_{3/2}}G_2=A_1u^n_1+2\lambda B_{1/2} u_l,\\
&-\lambda B_{j-1/2}e^{\phi^n_{j-1/2}}G_{j-1}+(A_je^{\phi^n_j}+\lambda B_{j+1/2}e^{\phi^n_{j+1/2}}+\lambda B_{j-1/2}e^{\phi^n_{j-1/2}})G_j-\lambda B_{j+1/2}e^{\phi^n_{j+1/2}}G_{j+1}=A_ju^n_j,\\
&-\lambda B_{N-1/2}e^{\phi^n_{N-1/2}}G_{N-1}+(A_N e^{\phi^n_N}+\lambda B_{N-1/2}e^{\phi^n_{N-1/2}}+2\lambda B_{N+1/2}e^{\phi^n_{N+1/2}})G_N=a_Nu^n_N+2\lambda B_{N+1/2} u_r.
\end{aligned}
\end{equation}
This linear system of $\{G_j\}$ admits a unique solution since its coefficient matrix  is strictly diagonally dominant.
Since $u^{n+1}_{j}=e^{\phi_j^n} G_j \geq e^{\phi_j^n} G_k$, where 
$$
G_{k}=\min_{1\leq j\leq N} \{ G_{j}\},
$$
it suffices to prove $G_k\geq0$.  We discuss in cases:   if $1<k<N,$ then from the $k$-th equation of (\ref{G}) with $A_k>0$ it follows 
\begin{align*}
A_ku^n_k  \leq  & -\lambda B_{k-1/2}e^{\phi^n_{k-1/2}}G_{k}+(A_ke^{\phi^n_k} +\lambda B_{k+1/2}e^{\phi^n_{k+1/2}}+\lambda B_{k-1/2}e^{\phi^n_{k-1/2}})G_k\\
&-\lambda B_{k+1/2}e^{\phi^n_{k+1/2}}G_{k}=A_ke^{\phi^n_k}G_k.
\end{align*}
Hence $G_{k} \geq u^n_k e^{- \phi^n_k} \geq 0$; if $k=1,$ from the first equation of (\ref{G}) we have
\begin{align*}
A_1u^n_1+2\lambda B_{1/2}u_l \leq & (A_1e^{\phi^n_1}+\lambda B_{3/2}e^{\phi^n_{3/2}}+2\lambda B_{1/2}e^{\phi^n_{1/2}})G_1-\lambda B_{3/2}e^{\phi^n_{3/2}}G_1= (A_1e^{\phi_1}+2\lambda B_{1/2}e^{\phi_l})G_1.
\end{align*}
This implies $G_1\geq 0$; so does the case if $k=N$.  

(ii) Likewise, scheme  (\ref{fully}), (\ref{UU}) and (\ref{Zf3}) can be rewritten as 
 \begin{equation*}
\begin{aligned}
&(A_1e^{\phi^n_1}+\lambda B_{3/2}e^{\phi^n_{3/2}})G_1-\lambda B_{3/2}e^{\phi^n_{3/2}}G_2=A_1u^n_1,\\
&-\lambda B_{j-1/2}e^{\phi^n_{j-1/2}}G_{j-1}+(A_je^{\phi^n_j}+\lambda B_{j+1/2}e^{\phi^n_{j+1/2}}+\lambda B_{j-1/2}e^{\phi^n_{j-1/2}})G_j-\lambda B_{j+1/2}e^{\phi^n_{j+1/2}}G_{j+1}=A_ju^n_j,\\
&-\lambda B_{N-1/2}e^{\phi^n_{N-1/2}}G_{N-1}+(A_N e^{\phi^n_N}+\lambda B_{N-1/2}e^{\phi^n_{N-1/2}})G_N=A_Nu^n_N.
\end{aligned}
\end{equation*}
Using an entirely same argument, we can show $G_j \geq 0$, hence $u_j^{n+1}\geq0$ for all $j$ involved. 
\end{proof}

\begin{rem} The specific values or choices of $\{\phi^n_j\}$ and $\{\phi^n_{j+1/2}\}$ do not affect the unconditional positivity property of the scheme for 
 $\{u^n_j\}.$  This result thus can be applied to the case when $\phi(x, t)$ is solved by the Poisson equation, see the next section.
\end{rem}

\section{Positive schemes for the reduced PNP-system }
The reduced PNP system (\ref{PNP2}) is reformulated as
\begin{equation}\label{PNP2+}
 \begin{aligned}
  A(x)  \partial_t c_i &= \partial_x (A(x)D_ie^{-z_i\psi} \partial_x (c_ie^{z_i\psi})),\\
   -\partial_x (\epsilon A(x) \partial_x \psi) & =  A(x) \left(\sum_{i=1}^{m}z_ic_i-\rho(x)\right).
\end{aligned}
\end{equation}
Let $c^n_{i,j}$ and $\psi^n_j$ approximate the cell average $\frac{1}{h}\int_{I_j}c_i(x,t_n)dx$ and $\frac{1}{h}\int_{I_j}\psi(x,t_n)$ respectively, then from the discretization strategy in section 2 the fully discrete scheme for system (\ref{PNP2+}) follows
\begin{align}
\label{fullyCi}
& A_j\frac{c^{n+1}_{i,j}-c^n_{i,j}}{\tau} =\frac{C_{i,j+1/2}-C_{i,j-1/2}}{h}, \\
\label{fullyPs}
& - \frac{\Psi^n_{j+1/2}-\Psi^n_{j-1/2}}{h}=A_{j}\bigg(\sum_{i=1}^{m}z_ic^n_{i,j}-\rho_j\bigg),
\end{align}
where numerical fluxes on interior interfaces are defined by 
\begin{align}
\label{CC}
C_{i,j+1/2} &=A_{j+1/2}D_i e^{-z_i\psi_{j+1/2}^n} \frac{\left( c^{n+1}_{i,j+1}e^{z_i\psi_{j+1}^n} -c^{n+1}_{i,j}e^{z_i\psi_{j}^n}\right)}{h}, \; j=1, \cdots, N-1.\\
\label{Ps2}
\Psi^n_{j+1/2}&=\epsilon A_{j+1/2}\frac{\psi^n_{j+1}-\psi^n_j}{h}, \quad j=1,\cdots, N-1,
\end{align}
where relevant terms are determined by 
\begin{align*}
& A_j=\frac{1}{h}\int_{I_j}A(x)dx,  \quad \rho_j  =\frac{1}{h}\int_{I_j}\rho(x)dx \\
& A_{j+1/2}=A(x_{j+1/2}), \; \psi^n_{j+1/2} =(\psi^n_{j}+\psi^n_{j+1})/2.
\end{align*}
For non-trivial $A(x), \rho(x)$, numerical integration of high accuracy is used to evaluate $A_j$ and $\rho_j$.  
The boundary fluxes are defined as follows: 

(i) for Dirichlet boundary condition (\ref{Di}), 
 \begin{equation}\label{BC11}
\begin{aligned}
 C_{i,1/2}&=A_{1/2}D_i\frac{2(c^{n+1}_{i,1}e^{z_i\psi^n_1}-c_{i,l})}{h}, \\
 C_{i,N+1/2}&=A_{N+1/2}D_ie^{-z_iV}\frac{2(c_{i,r}e^{z_iV}-c^{n+1}_{i,N}e^{z_i\psi^n_{N}})}{h},\\
 \Psi^n_{1/2}&=\epsilon A_{1/2}\frac{2\psi^n_{1}}{h}, \\
 \Psi^n_{N+1/2}&=\epsilon A_{N+1/2}\frac{2(V-\psi^n_{N})}{h},
 \end{aligned}
\end{equation}

(ii) for boundary condition (\ref{Zf}): \\
\begin{equation}\label{ZfBC}
\begin{aligned}
& C_{i,1/2}=0,\quad C_{i,N+1/2}=0,\\
& \Psi^n_{1/2}=\frac{\epsilon}{\eta} A_{1/2}(\psi^n_{1}-\psi_{_-}),\quad \Psi^n_{N+1/2}=\frac{\epsilon}{\eta} A_{N+1/2}(\psi_{_+}-\psi^n_{N}).
\end{aligned}
\end{equation}
\subsection{Scheme properties}
Scheme (\ref{fullyCi})-(\ref{Ps2}) with (\ref{BC11}) turns out to be unconditionally positivity-preserving. 
\begin{thm} Let  $\psi^n_j$ and $c_{i,j}^{n+1}$ for  $i=1,\cdots, m,$ $j=1, \cdots, N$ be obtained from (\ref{fullyCi})-(\ref{Ps2}) with (\ref{BC11}). If $c^n_{i,j}\geq 0$ and $c_{i,l}\geq 0,$ $c_{i,r}\geq 0$ for $i=1,\cdots, m$, $j=1, \cdots, N$,  then $c^{n+1}_{i,j}\geq 0$ for all $i=1,\cdots, m,$ $j=1, \cdots, N$.
\end{thm}
\begin{proof}
For fixed $i=1,\cdots,m$, the scheme (\ref{fullyCi}), (\ref{CC}) and (\ref{BC11}) is of the same form as (\ref{fully}), (\ref{UU}) and (\ref{U11}) with  $u^n_j=c^n_{i,j},$ $B_{j+1/2}=A_{j+1/2}D_i$, $\phi^n_j=-z_i\psi^n_j$ and $\phi^n_{j+1/2}=-z_i\psi^n_{j+1/2}$.  
 From (i) in Theorem 2.2, we can conclude $c^{n+1}_{i,j}=u^{n+1}_j \geq 0$.
\end{proof}
\begin{rem} From the above analysis we see that 
positivity of $c_{i,j}^n$ remains true even when another Poisson solver is used. 
\end{rem}

For scheme (\ref{fullyCi})-(\ref{Ps2}) with (\ref{ZfBC}), it turns out that the solution $c^n_{i,j}$ is conservative, non-negative, and energy dissipating.  
In order to state the energy dissipation result, we define a discrete version of the free energy (\ref{energy1}) as 
\begin{equation}\label{fullyenergy}
E_h^n=\sum_{j=1}^NhA_j\bigg(\sum_{i=1}^mc^n_{i,j}\log c^n_{i,j}+\frac{1}{2}S_j^n \psi^n_j \bigg)+\frac{\epsilon}{2\eta }(\psi_{_+}A_{N+1/2}\psi^n_{N}+\psi_{_-}A_{1/2}\psi^n_{1}),
\end{equation}
where 
$$
S_j^n=\sum_{i=1}^m z_ic^n_{i,j}-\rho_j.
$$
\begin{thm} Let $\psi^n_j$ and $c^{n}_{i,j}$ be obtained from (\ref{fullyCi})-(\ref{Ps2}) and (\ref{ZfBC}), then we have:\\
(1) Conservation of mass:
\begin{equation}\label{mass2}
\sum_{j=1}^NhA_jc_{i,j}^{n+1}= \sum_{j=1}^Nh A_jc_{i,j}^n \ \ \text{ for } n\geq 0, i=1,\cdots,m;
\end{equation}
(2) Propagation of positivity: if $c_{i,j}^n\geq0$ for all $j=1,\cdots, N,$ and $i=1,\cdots,m,$ then
$$c_{i,j}^{n+1}\geq 0, \quad j=1,\cdots,N, i=1,\cdots,m;
$$
(3) Energy dissipation: there exists $C^*>0$ depending on numerical solutions but independent on $\tau$ and $h$, 
such that if $\tau\leq C^*\epsilon/\eta$, then 
\begin{equation}\label{FE}
E_{h}^{n+1}-E_{h}^n\leq-\frac{\tau}{2}I_h^n,
\end{equation}
where
$$
I_h^n=\sum_{i=1}^m \sum_{j=1}^{N-1}\frac{1}{h}A_{j+1/2}D_i (c_{i,j+1}^{n+1}e^{z_i\psi_{j+1}^n}-c_{i,j}^{n+1}e^{z_i\psi_j^n})(  \log c_{i,j+1}^{n+1}e^{z_i\psi_{j+1}^n}-\log c_{i,j}^{n+1}e^{z_i\psi_j^n})\geq 0.
$$
\end{thm}

\begin{proof} (1) Mass conservation follows from summing (\ref{fullyCi}) over $j=1,\cdots, N$ and using (\ref{ZfBC}).

(2) For each fixed $i=1, \cdots m$, this follows from (ii) in Theorem 2.2,  by taking $u^n_j=c^n_{i,j}$, $B_{j+1/2}=A_{j+1/2}D_i$, $\phi^n_j=-z_i\psi^n_j$ and $\phi^n_{j+1/2}=-z_i\psi^n_{j+1/2}.$

(3) Using (\ref{fullyenergy}) we find that
\begin{align*}
E^{n+1}_h-E^n_h 
=&\sum_{j=1}^NhA_j\bigg(  \sum_{i=1}^m(c_{i,j}^{n+1}-c_{i,j}^n)(\log c_{i,j}^{n+1}+z_i\psi_j^n)+   \sum_{i=1}^m c_{i,j}^n\log c_{i,j}^{n+1} -\sum_{i=1}^m c_{i,j}^n\log c_{i,j}^n  \\
&+ \frac{1}{2}S_j^{n+1} \psi_j^{n+1}-\frac{1}{2}S_j^n \psi_j^n-\sum_{i=1}^m z_i(c_{i,j}^{n+1}-c_{i,j}^n)\psi_j^n\bigg)\\
& +\frac{\epsilon}{2\eta }(A_{N+1/2}\psi_{_+}\psi_N^{n+1}+A_{1/2}\psi_{_-}\psi_1^{n+1})-\frac{\epsilon}{2\eta}(A_{N+1/2}\psi_{_+}\psi_N^{n}+A_{1/2}\psi_{_-}\psi_1^{n})\\
=: &I+II+III. 
\end{align*}
We proceed to estimate term by term. 
For $I$, we use scheme (\ref{fullyCi})-(\ref{CC}) and (\ref{ZfBC}) and summation by parts to obtain
\begin{align*}
I=& \sum_{j=1}^NhA_j  \sum_{i=1}^m(c_{i,j}^{n+1}-c_{i,j}^n)(\log c_{i,j}^{n+1}+z_i\psi_j^n)\\
=& \tau  \sum_{i=1}^m \sum_{j=1}^N (C_{i,j+1/2}-C_{i,j-1/2}) \log (c_{i,j}^{n+1}e^{z_i\psi_j^n})\\ 
=& -\tau  \sum_{i=1}^m \sum_{j=1}^{N-1} C_{i,j+1/2}(  \log c_{i,j+1}^{n+1}e^{z_i\psi_{j+1}^n}-\log c_{i,j}^{n+1}e^{z_i\psi_j^n})\\ 
=& -\tau  \sum_{i=1}^m \sum_{j=1}^{N-1}\frac{1}{h}A_{j+1/2}D_i (c_{i,j+1}^{n+1}e^{z_i\psi_{j+1}^n}-c_{i,j}^{n+1}e^{z_i\psi_j^n})(  \log c_{i,j+1}^{n+1}e^{z_i\psi_{j+1}^n}-\log c_{i,j}^{n+1}e^{z_i\psi_j^n})\\
=& -\tau I_h^n\leq0.
\end{align*}
For $II$,  we use $\log(X)\leq X-1$ for $X>0,$ to obtain   
\begin{align*}
II=&\sum_{j=1}^{N}hA_{j} ( \sum_{i=1}^m c_{i,j}^n\log c_{i,j}^{n+1} -\sum_{i=1}^m c_{i,j}^n\log c_{i,j}^n   )\\
=& \sum_{j=1}^{N}hA_{j}  \sum_{i=1}^m c_{i,j}^n\log \frac{c_{i,j}^{n+1}}{c_{i,j}^n} \\
\leq &  \sum_{i=1}^m  \sum_{j=1}^{N}hA_{j} c_{i,j}^n( \frac{c_{i,j}^{n+1}}{c_{i,j}^n}-1) \\
= &  \sum_{i=1}^m  \sum_{j=1}^{N}hA_{j} ( c_{i,j}^{n+1} -c_{i,j}^n) =0,
\end{align*}
where in the last equality we have used conservation of mass.

Rearranging terms in $III$, we find that
\begin{align*}
III=&\sum_{j=1}^{N}hA_{j}\bigg(  \frac{1}{2}S_j^{n+1} \psi_j^{n+1}+\frac{1}{2}S_j^n\psi_j^n-S_j^{n+1}\psi_j^n\bigg)\\
& +\frac{\epsilon}{2\eta }(A_{N+1/2}\psi_{_+}\psi_N^{n+1}+A_{1/2}\psi_{_-}\psi_1^{n+1})-\frac{\epsilon}{2\eta}(A_{N+1/2}\psi_{_+}\psi_N^{n}+A_{1/2}\psi_{_-}\psi_1^{n})\\
&=\frac{1}{2}\sum_{j=1}^{N}hA_{j}(S_j^{n+1}-S_j^n)(\psi_j^{n+1}-\psi_j^n)+\mathbf{F},
\end{align*}
where 
\begin{align*}
\mathbf{F}=&\frac{1}{2} \sum_{j=1}^{N}hA_{j}(   S_j^n \psi_j^{n+1}-S_j^{n+1}\psi_j^n)+\frac{\epsilon}{2\eta }(A_{N+1/2}\psi_{_+}\psi_N^{n+1}+A_{1/2}\psi_{_-}\psi_1^{n+1})\\
&-\frac{\epsilon}{2\eta }(A_{N+1/2}\psi_{_+}\psi_N^{n}+A_{1/2}\psi_{_-}\psi_1^{n}).
\end{align*}
Tedious but elementary calculations show that 
$\mathbf{F} \equiv 0.$ Thus 
\begin{equation}\label{C4}
III=\frac{1}{2}\sum_{j=1}^{N}hA_{j}(S_j^{n+1}-S_j^n)(\psi_j^{n+1}-\psi_j^n).
\end{equation}
Scheme (\ref{fullyPs})-(\ref{Ps2}) and (\ref{ZfBC}) can be written in matrix form 
$$
M\vec{\psi}^n=\vec{b},
$$ 
where
\begin{equation}\label{TrM}
M= \begin{bmatrix}
\frac{h}{\eta} A_{1/2}+A_{3/2} & -A_{3/2} &  &  &  \\
-A_{3/2} & A_{3/2}+A_{5/2} & -A_{5/2} &  &  \\
 &\ddots &\ddots&\ddots &  \\
 &  &  -A_{N-3/2} &  A_{N-3/2}+A_{N-1/2} &  -A_{N-1/2} \\
 &  &  &  -A_{N-1/2} & \frac{h}{\eta} A_{N+1/2}+A_{N-1/2}
\end{bmatrix},  
\end{equation}
$$
\vec{b}=\frac{h^2}{\epsilon}\bigg(A_1S_1^n +\frac{\epsilon}{h\eta }A_{1/2}\psi_{_-}, A_2S_2^n, \cdots, A_NS_N^n+\frac{\epsilon}{h \eta}A_{N+1/2}\psi_{_+} \bigg)^\top.
$$
Hence we have 
$$
\psi_j^{n+1}-\psi_j^n=\frac{\tau h^2}{\epsilon}\sum_{k=1}^N(M^{-1})_{j,k}A_k D_t S_k^n, \quad \tau D_t S_j^n: =S_j^{n+1}-S_j^n,
$$
thus (\ref{C4})  can be simplified as
$$
III=\frac{h^3\tau^2 }{2\epsilon} \sum_{j=1}^NA_jD_tS_j^n\sum_{k=1}^N(M^{-1})_{j,k}A_kD_tS_k^n.
$$
We claim that for any $\zeta \in \mathbb{R}^N$ 
\begin{align}\label{em}
\zeta \cdot M^{-1} \zeta \leq \frac{N^2 \eta}{ (A_{1/2}+A_{N+1/2})}\|\zeta\|^2,
\end{align}
with which we can bound $III$ as
\begin{equation}\label{C5}
\begin{aligned}
III=& \frac{h^3\tau^2 }{2\epsilon} \sum_{j=1}^NA_jD_tS_j^n\sum_{k=1}^N(M^{-1})_{j,k}A_kD_t S_k^n\\
\leq & \frac{\alpha \eta N^2 h^3\tau^2 }{2\epsilon} \sum_{j=1}^NA^2_j |D_t S_j^n|^2,
\end{aligned}
\end{equation}
where $\alpha^{-1}=A_{1/2}+A_{N+1/2}$.
Note that $hN=1$ and 
$$
|D_t S_j^n|^2\leq m\sum_{i=1}^mz_i^2(D_tc_{i,j}^n)^2,
$$
we thus have
$$
III \leq \sum_{i=1}^m\frac{\alpha \eta z_i^2m\tau^2}{2\epsilon}\sum_{j=1}^NhA_j^2 (D_tc_{i,j}^n)^2.
$$
Collecting  estimates on $I,II$ and $III$ we arrive at 
$$
E_h^{n+1}-E_h^n\leq \sum_{i=1}^m \Bigg( \tau \sum_{j=1}^NhA_j  (D_tc_{i,j}^n) (\log c_{i,j}^{n+1}+z_i\psi_j^n)   +  \frac{\alpha \eta z_i^2m\tau^2}{2\epsilon}\sum_{j=1}^NhA_j^2 (D_tc_{i,j}^n)^2\Bigg).
$$
For (\ref{FE})  to hold, it remains to find a sufficient condition on time step $\tau$ so that for all $i=1,\cdots,m$, 
\begin{equation}\label{Final}
 \frac{\alpha\eta  z_i^2m\tau^2}{2\epsilon}\sum_{j=1}^NhA_j^2 (D_tc_{i,j}^n)^2 \leq -\frac{\tau}{2} \sum_{j=1}^NhA_j  (D_tc_{i,j}^n)(\log c_{i,j}^{n+1}+z_i\psi_j^n).
\end{equation}
This is nothing but 
$$
\frac{\alpha\eta z_i^2m\tau^2}{2\epsilon}\|\vec{\xi}\|^2+\frac{\tau}{2} \vec{\xi}\cdot \vec{\mu}\leq 0,
$$
where 
$$
\vec{\xi}_j=\sqrt{h}A_jD_t c_{i,j}^{n}, \; \vec{\mu}_j=\sqrt{h}(\log c_{i,j}^{n+1})+z_i\psi^n_j).
$$
Note that $I=\tau \vec{\xi}\cdot \vec{\mu} \leq 0$. One can verify using (\ref{CC}) and flux (\ref{fullyCi}) that $\vec{\xi}\cdot \vec{\mu}=0$ if and only if $ \vec{\xi}=0$. 
Therefore 
$$
0< c_0 \leq \frac{-\vec{\xi}\cdot \vec{\mu}}{\|\vec{\xi}\|^2}\leq \frac{\|\vec{\mu}\|}{\|\vec{\xi}\|} \quad \text{for} \; \vec{\xi} \not=0,
$$
where $c_0$ depends on the numerical solution at $t_n$ and $t_{n+1}$. We thus obtain (\ref{Final}) by taking
$$
\tau\leq C^*\frac{\epsilon}{\eta} , \quad \text{where  \ \ } C^*=\min_{1\leq i\leq m} \frac{c_0}{\alpha z_i^2m}>0.
$$
Finally, we return to the proof of claim (\ref{em}):
For any $y\in \mathbb{R}^N$ with $\|y\|=1$, we have the following 
\begin{align*}
y\cdot M y& =\frac{h}{\eta} A_{1/2}y_1^2+ \sum_{j=1}^{N-1}A_{j+1/2}(y_{j+1}-y_j)^2 +\frac{h}{\eta} A_{N+1/2}y_N^2 \\
& \geq \min_{\|y\|=1} \{ \frac{h}{\eta} A_{1/2}y_1^2+ \sum_{j=1}^{N-1}A_{j+1/2}(y_{j+1}-y_j)^2 +\frac{h}{\eta} A_{N+1/2}y_N^2\}\\
&  =\frac{ h}{N \eta} (A_{1/2}+A_{N+1/2}), 
\end{align*}
where the minimum is achieved at $y=(1, \cdots, 1)/\sqrt{N}$. Replacing $y$ by $y/\|y\|$ and then further set $y=M^{-1/2}\zeta$ leads to (\ref{em}).
\end{proof}
\begin{rem} Though $C^*$ is not explicitly given, it is about $O(1)$ as can be seen from a formal limit $\Delta t \to 0$.
The sufficient condition $\tau \leq C^* \epsilon/\eta$ suggests that for smaller $\epsilon/\eta $, one should consider a smaller time step to ensure the scheme stability.  This is consistent with our numerical results. 
\end{rem}


\section{Numerical tests}
In this section, we implement the fully discrete scheme (\ref{fullyCi})-(\ref{Ps2}) with different boundary conditions. Errors are measured in the following discrete $l^{\infty}$ norm:  
$$
e_f=\max_{1\leq j \leq N}| f_j-\bar{f}_{j}|.
$$
Here $\bar{f}_{j}$ denotes the average of $f$ on cell $I_j$.  In what follows we take $f_j=c_{i,j}^n$ or $\psi_{j}^n$ at time $t=n\tau.$

\subsection{Accuracy test} In this example we numerically verify the accuracy and order of schemes (\ref{fullyCi})-(\ref{Ps2}) with first order boundary flux (\ref{BC11}) and second order boundary flux of form (\ref{U22}). 
\begin{example}\label{ex1}
Consider the initial value problem with source term
 
 \begin{equation}
 \left \{
\begin{array}{rl}
 \hfill  \partial_t c_1 =&\frac{1}{A(x)}\partial_x(A(x)D_1(\partial_x c_1 +  z_1c_1 \partial_x \psi)+f_1(x,t),  \hfill \ \ \  x\in [0, \ 1], \ t>0,\\
 \hfill  \partial_t c_2 =&\frac{1}{A(x)}\partial_x(A(x)D_2(\partial_x c_2 +  z_2c_2 \partial_x \psi)+f_2(x,t),  \hfill \ \ \   x\in [0, \ 1], \ t>0,\\
 \hfill  -\frac{1}{A(x)}&\partial_x(\epsilon A(x) \partial_x \psi) =z_1c_1 +  z_2c_2 -\rho(x)+f_3(x,t),  \hfill \ \ \  x\in [0, \ 1], \ t>0,\\ 
  \hfill c_1(x,0)=&x^2(1-x) , \quad   c_1(0,t)=c_1(1,t)=0, \\
  \hfill c_2(x,0)=&x^2(1-x)^2 , \quad   c_2(0,t)=c_2(1,t)=0, \\
    \hfill \psi(0,t)=&0 , \quad   \psi(1,t)=-\frac{1}{60}e^{-t}.
\end{array}
\right.
\end{equation}
Here we take $A(x)=(5-4x)^2,$ $D_1=D_2=1,$ $z_1=-z_2=1,$ $\epsilon=1$ and $\rho(x)=0,$  source terms are 
\begin{align*}
f_1(x,t)&=\frac{4x^4 -9x^3+53x^2 -54x+10}{4x-5}e^{-t}+\frac{40x^7-71x^6+30x^5}{20}e^{-2t},\\
f_2(x,t)&=\frac{4x^5-13x^4 +94x^3-161x^2 +84x-10}{5-4x}e^{-t}+\frac{22x^8-60x^7+53x^6-15x^5}{10}e^{-2t},\\
f_3(x,t)&=-\frac{2x^4}{5}e^{-t}.
\end{align*}
The exact solution to (\ref{ex1}) is 
$$
c_1(x,t)=x^2(1-x)e^{-t}, \quad c_2(x,t)=x^2(1-x)^2e^{-t}, \quad \text{and} \;  \psi(x,t)=-\frac{x^5(3-2x)}{60}e^{-t}.
$$ 
We use the time step $\tau=h^2 $ to compute numerical solutions. The errors and orders at $t=1$ are listed in Table 1 and Table 2.
 
  \begin{table}[ht]\label{ex11}
        \centering
                \caption{\tiny{Accuracy for Example \ref{ex1} with first order boundary approximations (\ref{BC11}) }}
        \begin{tabular}{|c| c |c |c| c |c|c| c |}
            \hline
              N& \ $c_1$ error & order& $c_2$ error &order &$\psi$ error &order \\ [0.5ex] 
                          \hline
          40 &   0.11184E-03&-&  0.57759E-04&- &0.83275E-05&-\\
            80 &  0.28354E-04&  1.9798 &   0.14407E-04&  2.0033  &0.20810E-05&2.0006  \\
            160 &   0.71370E-05&  1.9902 &   0.36019E-05& 1.9999   & 0.52013E-06& 2.0003 \\
            320 &   0.17903E-05&  1.9951&   0.90047E-06&  2.0000  &0.13002E-06& 2.0001  
             \\ [1ex]
            \hline
        \end{tabular}
     \end{table}

    \begin{table}[ht]\label{ex121}
        \centering
                 \caption{\tiny{Accuracy for Example \ref{ex1} with second order boundary approximations (\ref{U22}) }}
        \begin{tabular}{|c| c |c |c| c |c|c| c |}
            \hline
              N& \ $c_1$ error & order& $c_2$ error &order &$\psi$ error &order \\ [0.5ex] 
                          \hline
          40 &   0.10014E-03&-&  0.69633E-04&- &0.37021E-05&-\\
            80 &  0.25204E-04&  1.9903 &   0.18005E-04&  1.9514 &0.93954E-06&1.9783  \\
            160 &   0.63218E-05&  1.9952 &   0.45767E-05& 1.9760   & 0.23755E-06& 1.9837 \\
            320 &   0.15830E-05&  1.9977&   0.11536E-05&  1.9882  &0.59655E-07& 1.9935
             \\ [1ex]
            \hline
        \end{tabular}                  
     \end{table}
     We see from Table 1 and Table 2 that both first and second order boundary fluxes yield second order convergent solutions. The numerical errors with both fluxes are comparable.

In the remaining numerical tests we only use first order boundary flux (\ref{BC11}) for the Dirichlet boundary value problem.
\end{example}

 \subsection{Effects of permanent charge and channel geometry}  The key structure of an ion channel includes both the channel shape and the permanent charge (see e.g. \cite{WS15}).  
 We present numerical examples to illustrate the effects from the channel geometry or the permanent change.   While we also examine dependence of the total current (\ref{J}) on voltage $V$, which is known as the current-voltage (I-V) relation in \cite{WS15}.  
 Note that (\ref{J}) can be reformulated as
$$
J=-\sum_{i=1}^m z_iD_iA(x) e^{-z_i\psi}\partial_x (c_ie^{z_i\psi}).
$$
 Let $J^n_{j+1/2}$ be an approximation of $J(x_{j+1/2},t_n),$ 
then $J^n_{j+1/2}$ can be computed by
$$
J^n_{j+1/2}=-\sum_{i=1}^m z_iC_{i,j+1/2}, 
$$
where $C_{i,j+1/2}$ is defined in (\ref{CC}) with $c^{n+1}_{i, j}$  replaced by $c^n_{i, j}$, that is, 
 $$
 C_{i,j+1/2} =A_{j+1/2}D_i e^{-z_i\psi_{j+1/2}^n} \frac{\left( c^{n}_{i,j+1}e^{z_i\psi_{j+1}^n} -c^{n}_{i,j}e^{z_i\psi_{j}^n}\right)}{h}, \; j=1, \cdots, N-1.
 $$

 \begin{example}\label{ex2} (Effects of channel geometry with permanent charge) We consider the system 
\begin{equation}\label{PNPQ}
 \begin{aligned}
  A(x)  \partial_t c_1 &= \partial_x (A(x)( \partial_x c_1 +c_1 \partial_x \psi)), \quad x\in [0, \ 1], \ t>0,\\
    A(x)  \partial_t c_2 &= \partial_x (A(x)( \partial_x c_2 -c_2 \partial_x \psi)), \quad x\in [0, \ 1], \ t>0,\\
   -\frac{1}{A(x)}\partial_x (\epsilon A(x) \partial_x \psi) & = c_1-c_2 -\rho(x), \quad x\in [0, \ 1], \ t>0,
\end{aligned}
\end{equation}
where $\epsilon=5\times 10^{-5},$ subject to initial and boundary conditions
\begin{equation}\label{IB}
\begin{aligned}
  c_1(x,0)& =c_2(x, 0) =0.5-0.1x, \quad x \in [0, 1],  \\
   c_i(0,t) &=0.5, \quad c_i(1,t)=0.4; \;\psi(0,t)=0, \psi(1,t)=0.5,\quad    t>0.
\end{aligned}
\end{equation}

\begin{figure}[!tbp]
\caption{Diagram of 1D computational region for the channel and bath funnels \cite{Eis04}}
  \centering
     \includegraphics[width=0.5\textwidth]{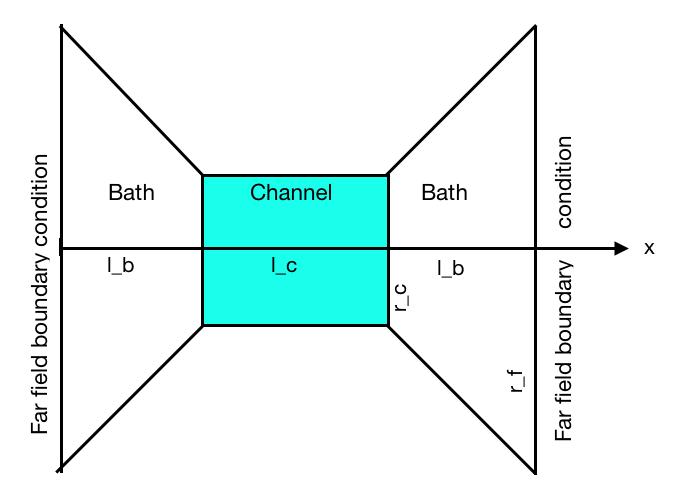}
\end{figure}
This corresponds to problem (\ref{PNP2}) with  $D_1=D_2=1$,  $z_1=-z_2=1$, with $c_{i,l}=0.5$, $c_{i,r}=0.4$, and $V=0.5$. 

The computational domain diagram is given in Figure 1,  while the cross sectional area $A(x)$  is defined as:
\begin{equation}\label{2Ax}
A(x)=\left \{
\begin{array}{rl}
 2(r_f+\frac{r_c-r_f}{l_b}x),  \quad   \hfill & x \in [0,\ l_b], \\\\
  2r_c, \quad   \hfill & x \in (l_b,\ l_b+l_c),  \\\\
  2( r_c+\frac{r_f-r_c}{l_b}(x-l_b-l_c)),  \hfill &  x \in [l_b+l_c,\ 1],
    \end{array}
\right.
\end{equation}
where the shape parameters are allowed to vary in our numerical tests. The permanent charge $\rho(x)$ is taken as
\begin{equation}\label{2Qx}
\rho(x)=\left \{
\begin{array}{rl}
 0,  \quad   \hfill & x \in [0,\ l_b], \\\\
  2Q_0 \quad   \hfill & x \in (l_b,\ l_b+l_c),  \\\\
  0,  \hfill &  x \in [l_b+l_c,\ 1], 
    \end{array}
\right.
\end{equation}
with $Q_0$ a fixed constant.  

Robert Eisenberg made clear to us the great importance of the tapered representation of the baths in one dimensional versions of PNP models of channels, that became clear in his early work with Wolfgang Nonner \cite{NE98, NCE98}, followed by many other more formal treatments such 
as in  \cite{Eis04}. 

In this numerical test, we take $h=0.01,$ $\tau=5\times 10^{-5}$.  The solutions are understood to have reached steady states if $||\psi^{n}-\psi^{n-1}||_{\infty}\leq 10^{-6}$. Table 3 shows times $t_s$ needed for reaching each steady state, number of iterations, and CPU times. 

    \begin{table}[ht]\label{ex122}
        \centering
                 \caption{\tiny{Times needed for reaching each steady state in Example \ref{ex2} when $Q_0=0.2, r_f=20$ with different channel geometry}}
        \begin{tabular}{|c| c |c |c| c |c|c| c |}
            \hline
              channel parameters &   $||\psi^{n}-\psi^{n-1}||_{\infty}$ &  time $t_s$  &  iterations $n=t_s/\tau$  &  CPU time (sec) \\ [0.5ex] 
            \hline
            $r_c=l_c=\frac{1}{3}$ &   9.9822E-07&    0.0744&    1488 & 0.5244 \\
            \hline
          $r_c=l_c=\frac{1}{5}$  & 9.9879E-07 & 0.0992&  1984 & 0.6534\\
          \hline
            $r_c=l_c=\frac{1}{11}$ & 9.9920E-07 & 0.1116  &   2232 &  0.7035 
             \\ [1ex]
            \hline
        \end{tabular}                  
     \end{table}
     
From Table 3 we see that $t_s=0.1116$ is the longest time needed for reaching the steady state, so we run the simulation up to $t=0.2$.

In Figure 2 we take $Q_0=0.2$, $r_f=20$, varying $l_c$ and $r_c$  inside the channel, to obtain a series of snapshots. We see that both $c_1$ and $c_2$ coincide outside the channel, but split inside the channel with the shape evolving in terms of the channel geometry.  The profile of $\psi$ looks similar.  

In Figure 3 we fix the channel shape with $r_f=20,$ $r_c=1/5,$ $l_c=1/5$, varying $Q_0$, we observe that the difference between $c_1$ and $c_2$ inside the channel increases in terms of $Q_0$, roughly we have $c_1-c_2\approx 2Q_0$ inside the channel. We can also observe the effects on $\psi$. 

In Figure 4 is the I-V relation for the PNP system with channel shape parameters  $l_c=1/5,$ $r_c=1/5,$ $r_f=20$, and $Q_0=0.1$.  We see from 
the figure that the current is linear in the voltage.

\begin{figure}[!tbp]
\caption{Effects of channel geometry on steady state densities and potential with $Q_0=0.2$: (a)-(c) densities at $t=0.2,$ for  $l_c=r_c=\frac{1}{3},\ \frac{1}{5}$ and $\frac{1}{11}$, (d)-(f) potential profiles at $t=0.2$ for  $l_c=r_c=\frac{1}{3},\ \frac{1}{5}$ and $\frac{1}{11}$. }
\centering  
\subfigure{\includegraphics[width=0.30\linewidth]{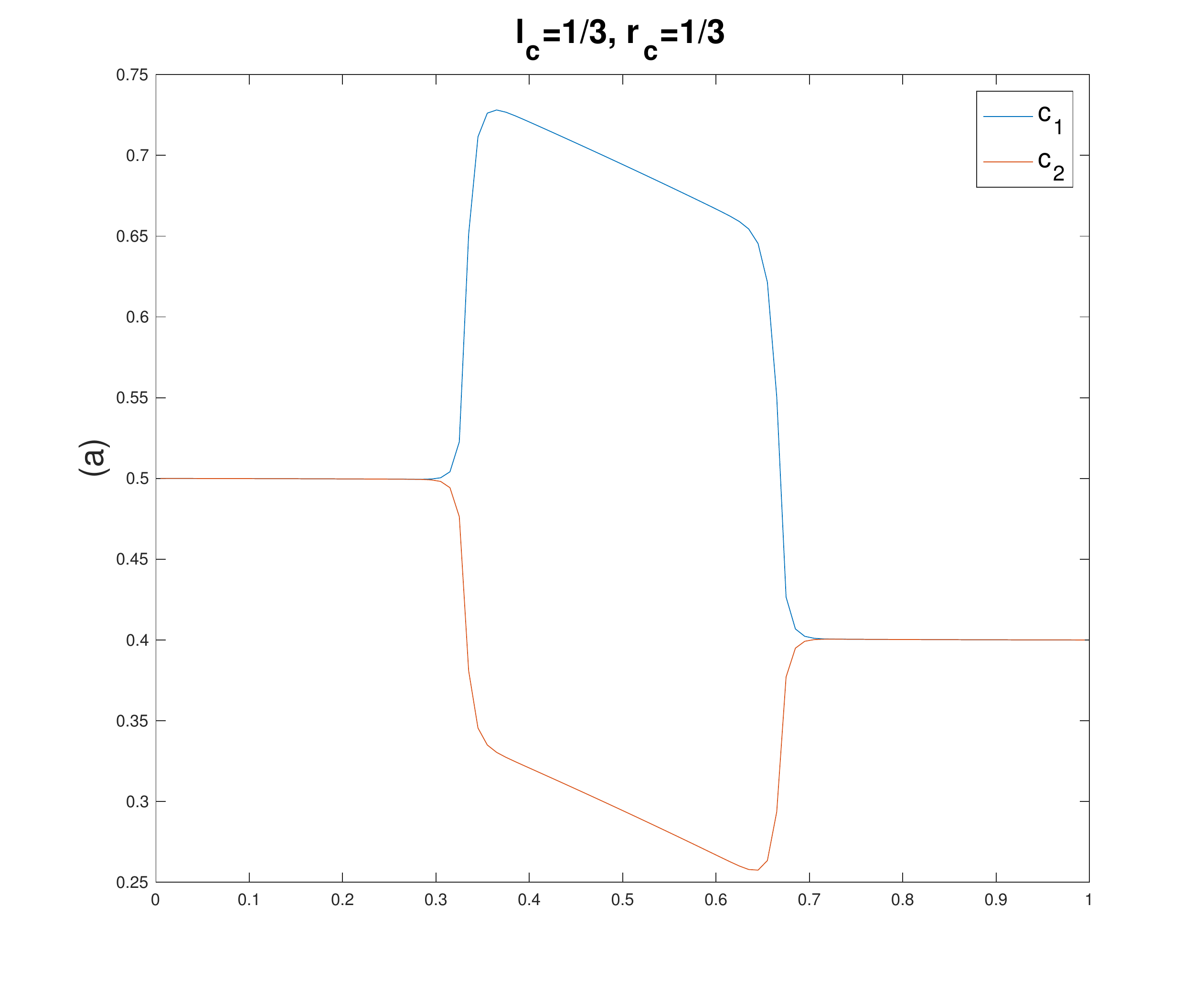}}
\subfigure{\includegraphics[width=0.30\linewidth]{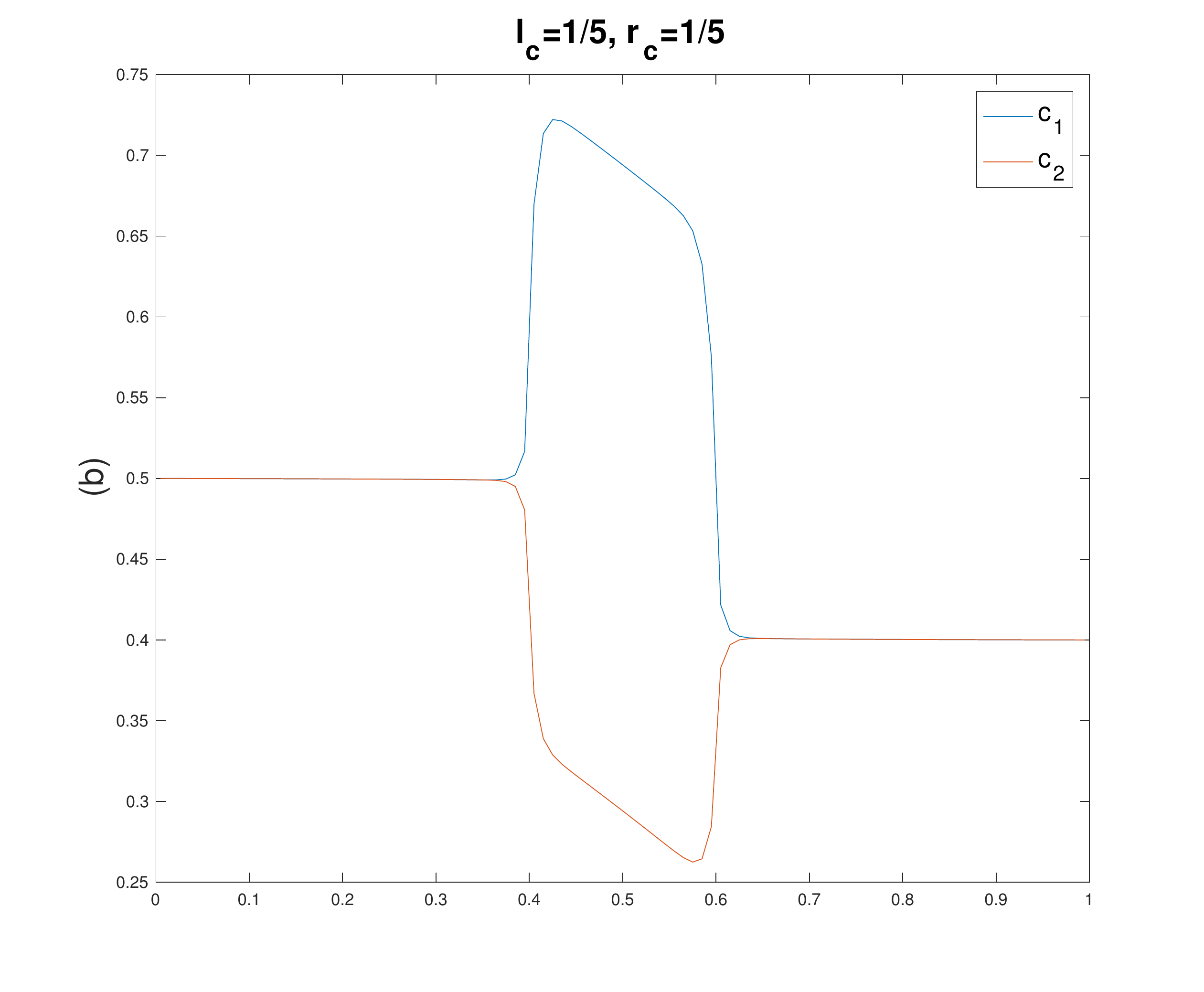}}
\subfigure{\includegraphics[width=0.30\linewidth]{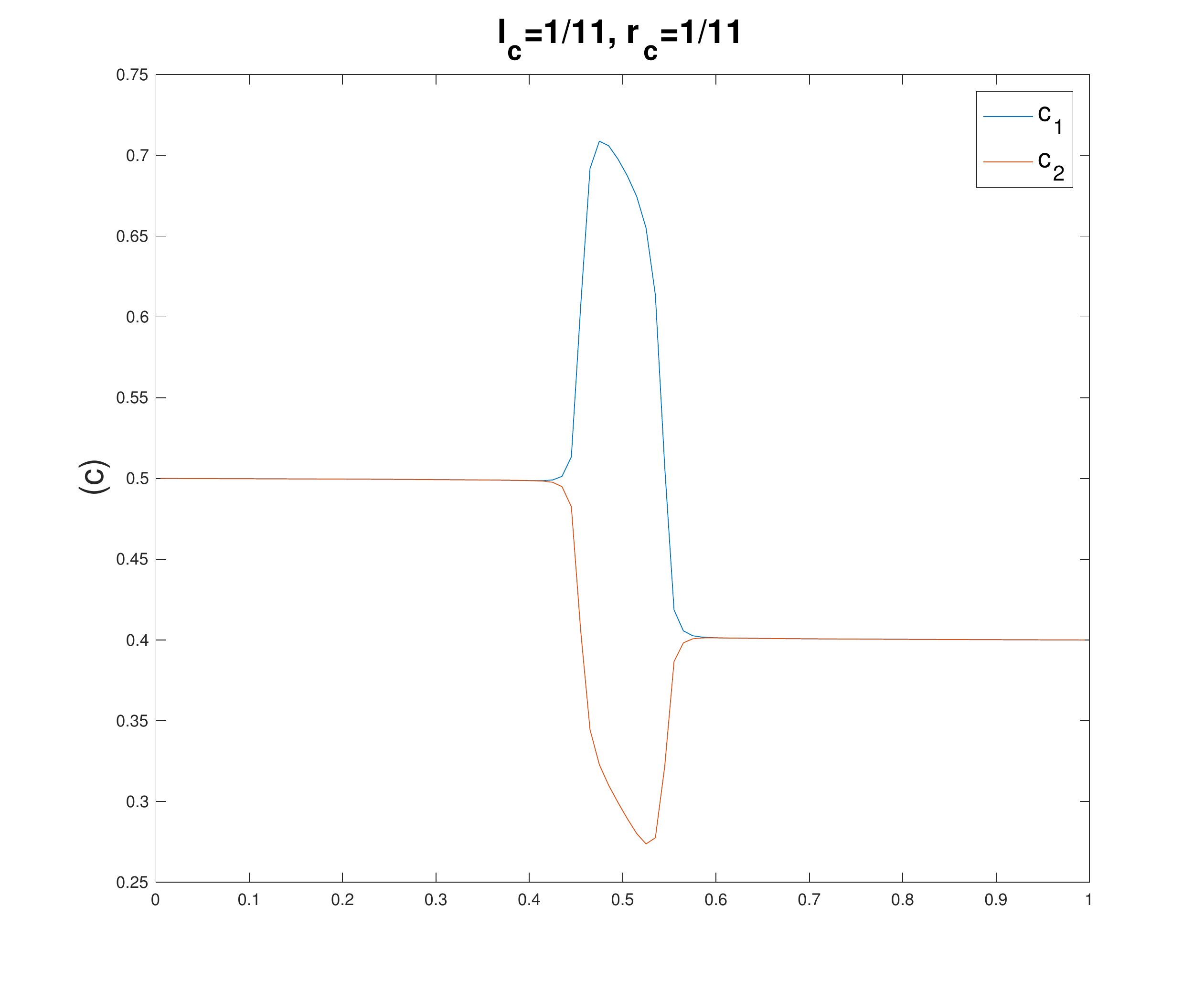}}\\
\subfigure{\includegraphics[width=0.30\linewidth]{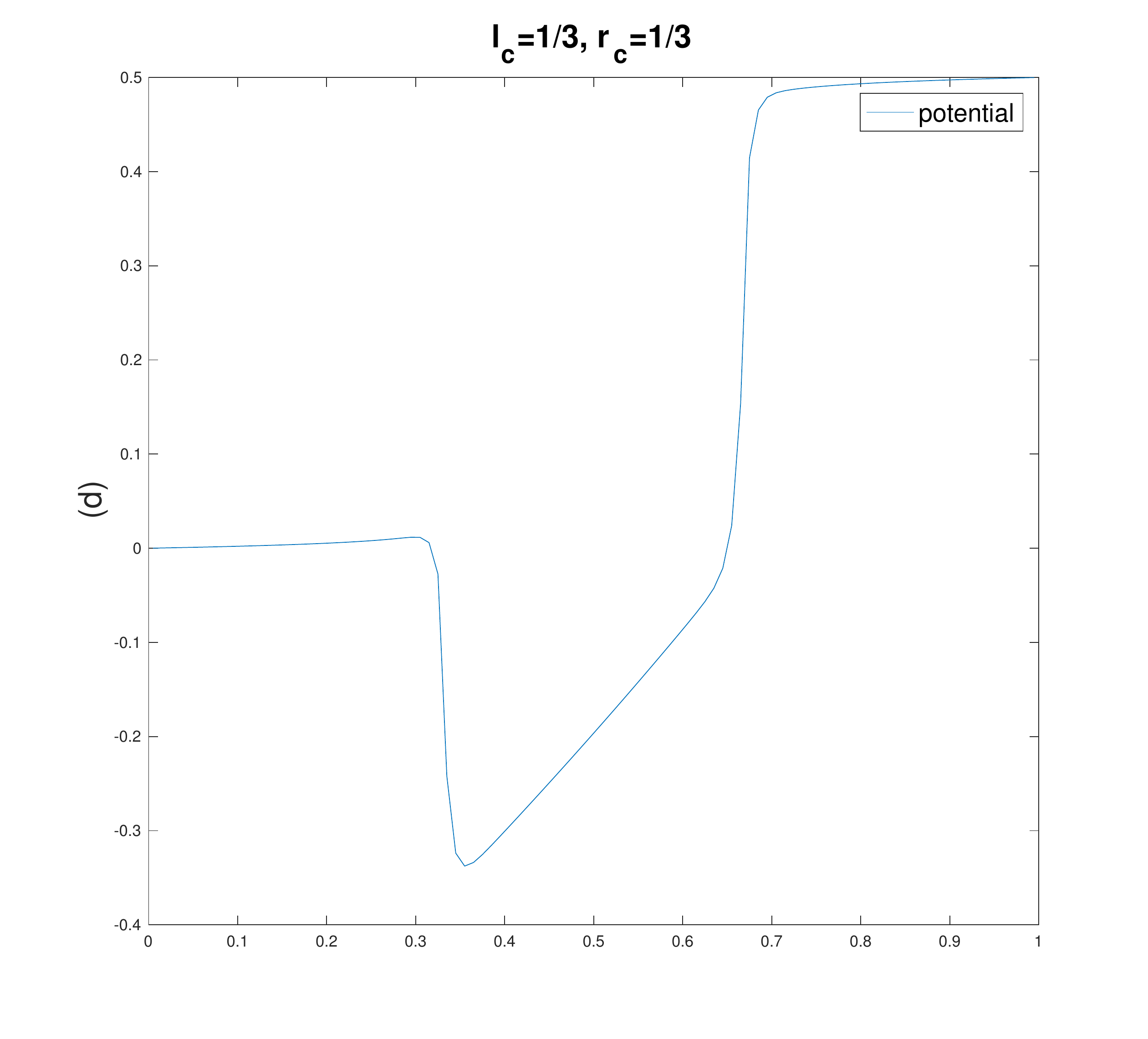}}
\subfigure{\includegraphics[width=0.30\linewidth]{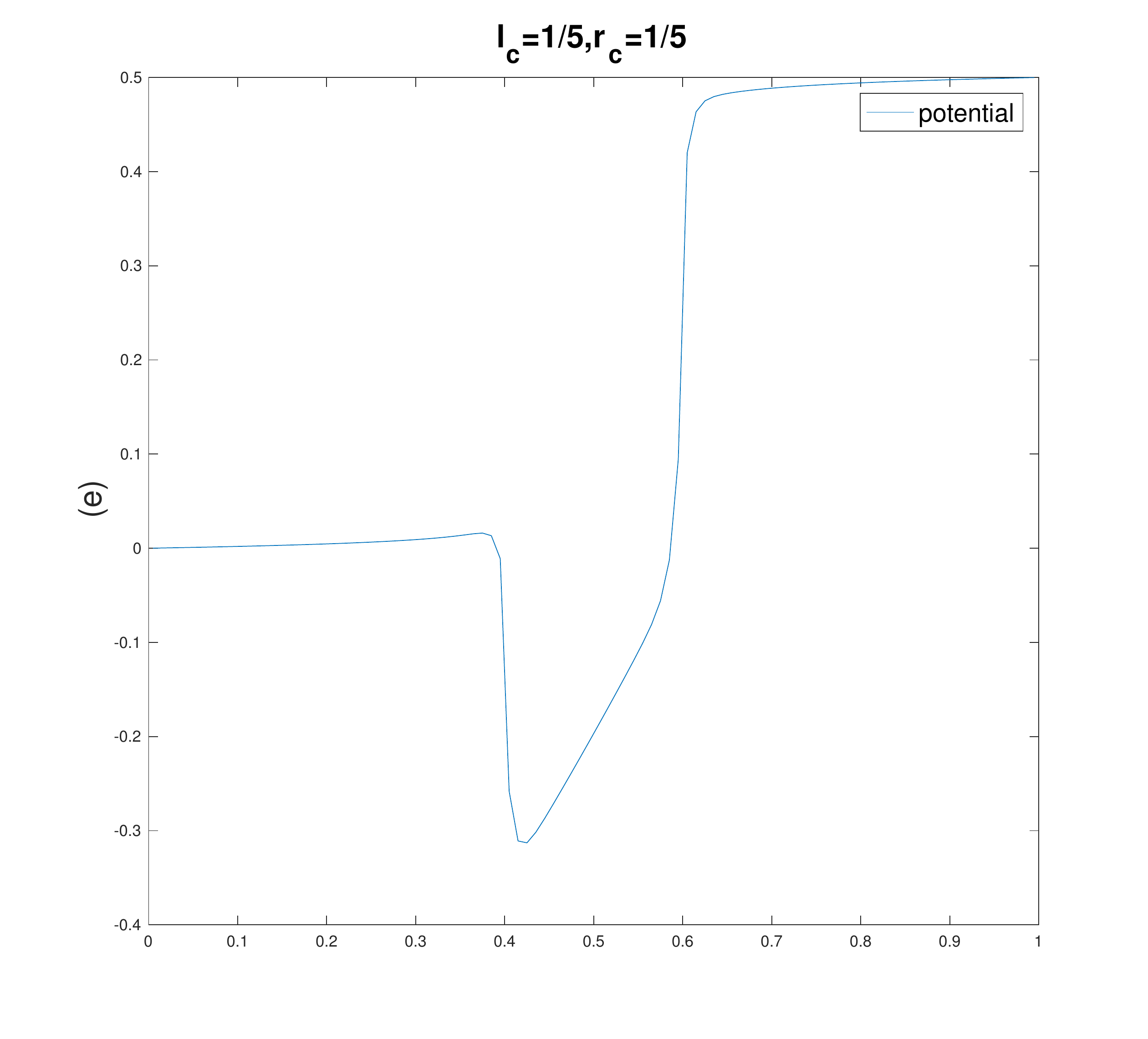}}
\subfigure{\includegraphics[width=0.30\linewidth]{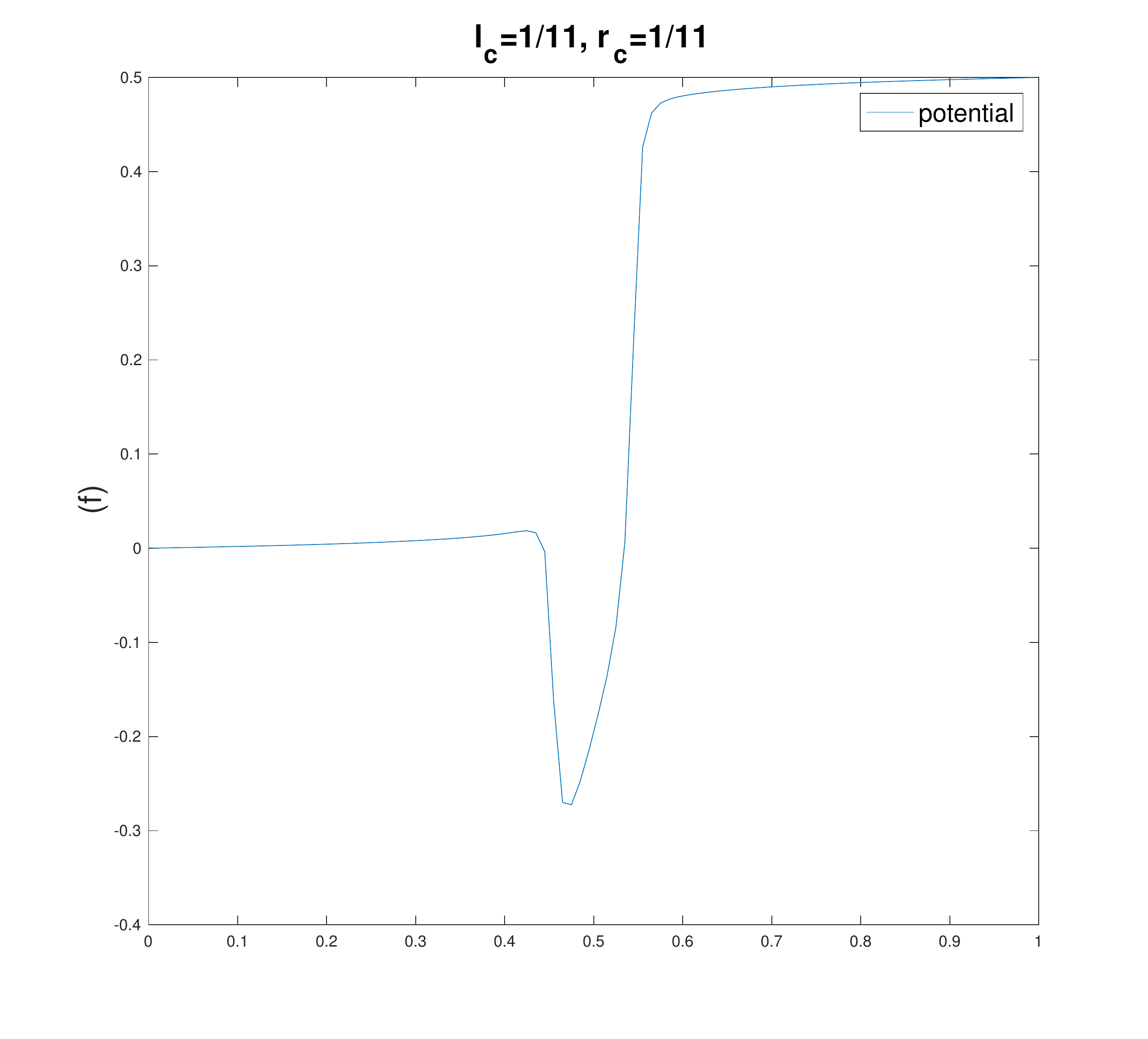}}
\end{figure}

\begin{figure}[!tbp]
\centering  
\caption{Effects of permanent charge on steady state densities and potential with $l_c=r_c=\frac{1}{5}$: (a)-(c) are computed densities at $t=0.2,$ for  $Q_0=0.05,\ 0.1,\ 0.15$, (d)-(f) are potential profiles at $t=0.2$ for  $Q_0=0.05,\ 0.1,\ 0.15$. }
\subfigure{\includegraphics[width=0.30\linewidth]{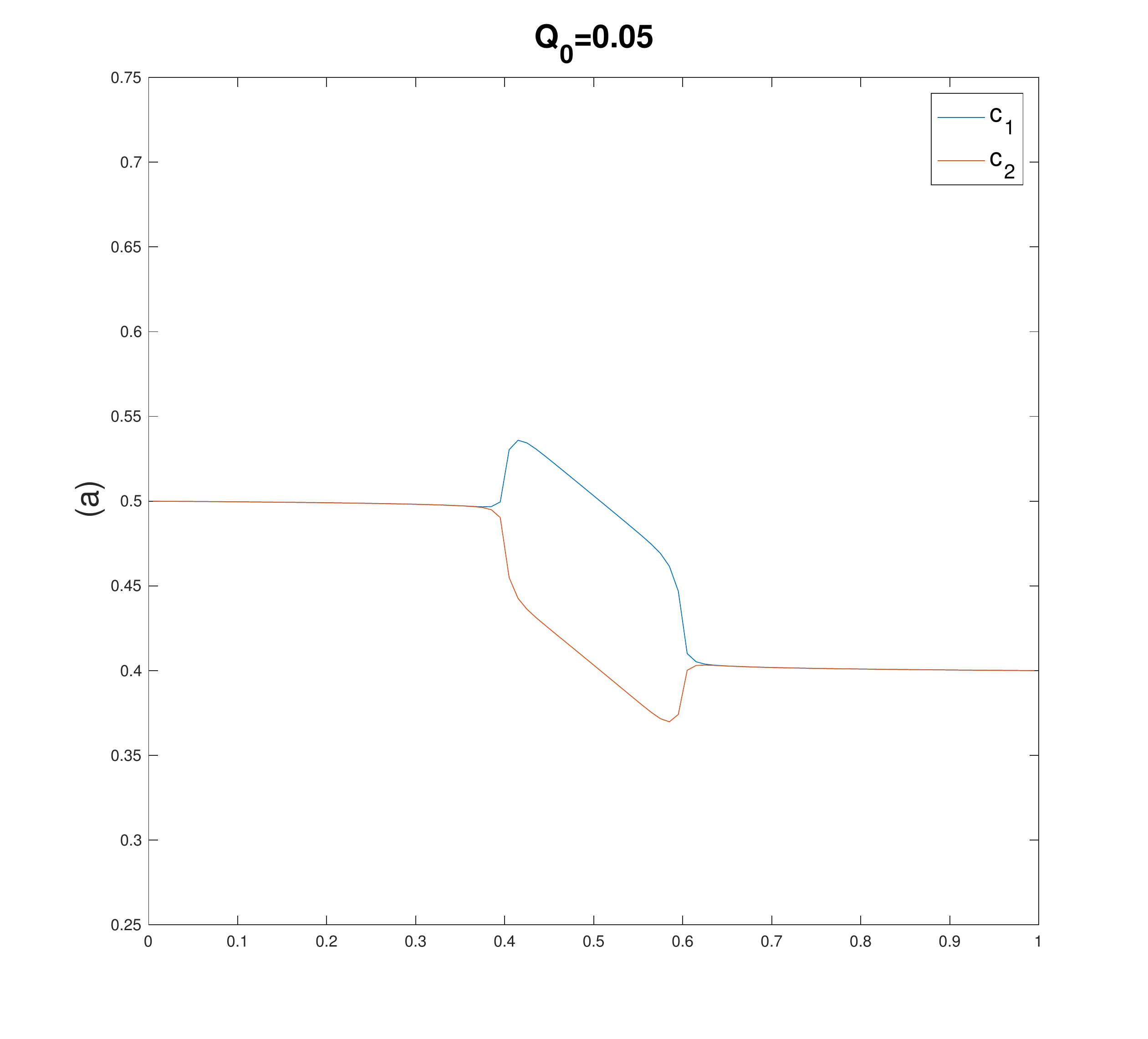}}
\subfigure{\includegraphics[width=0.30\linewidth]{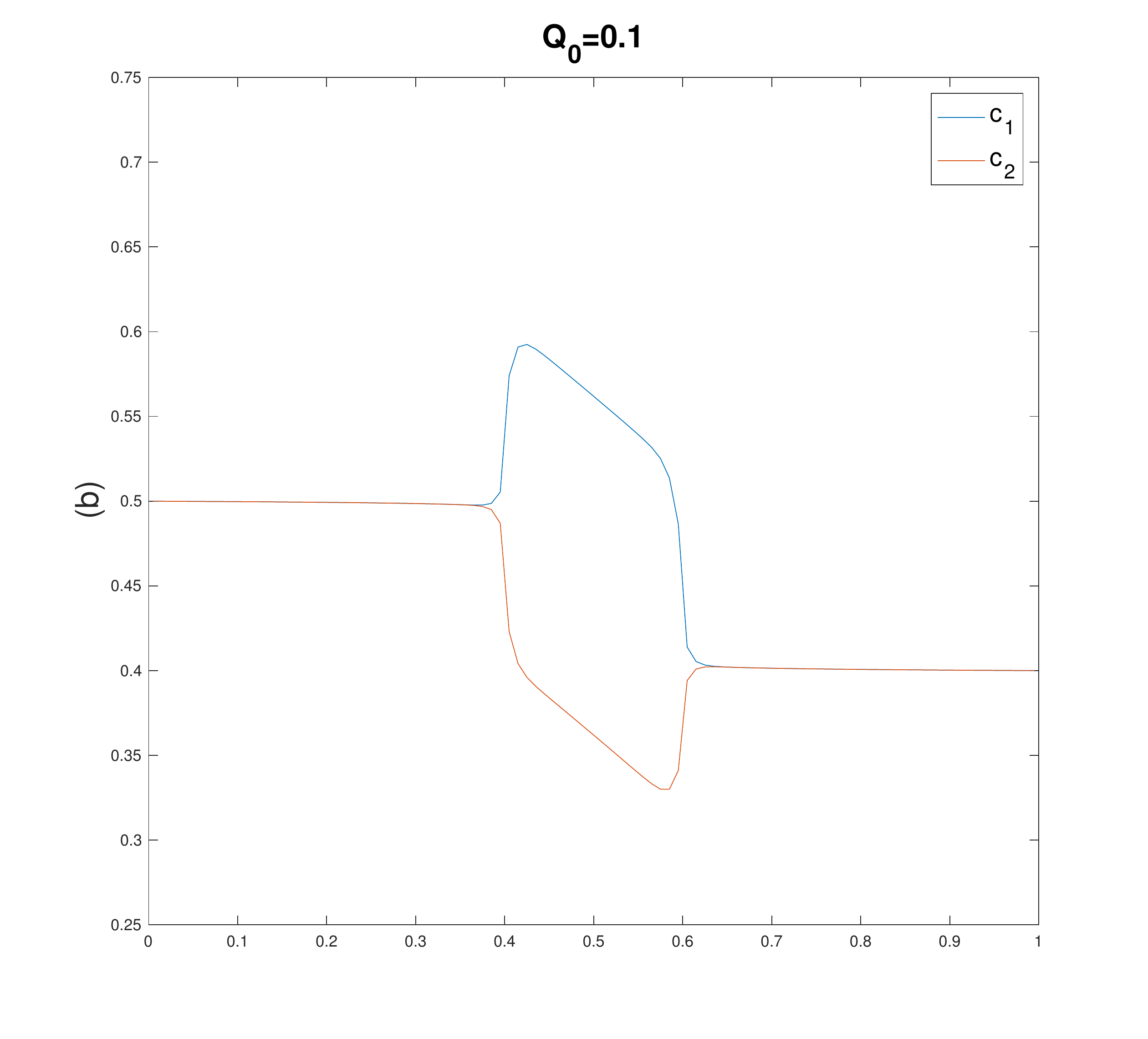}}
\subfigure{\includegraphics[width=0.30\linewidth]{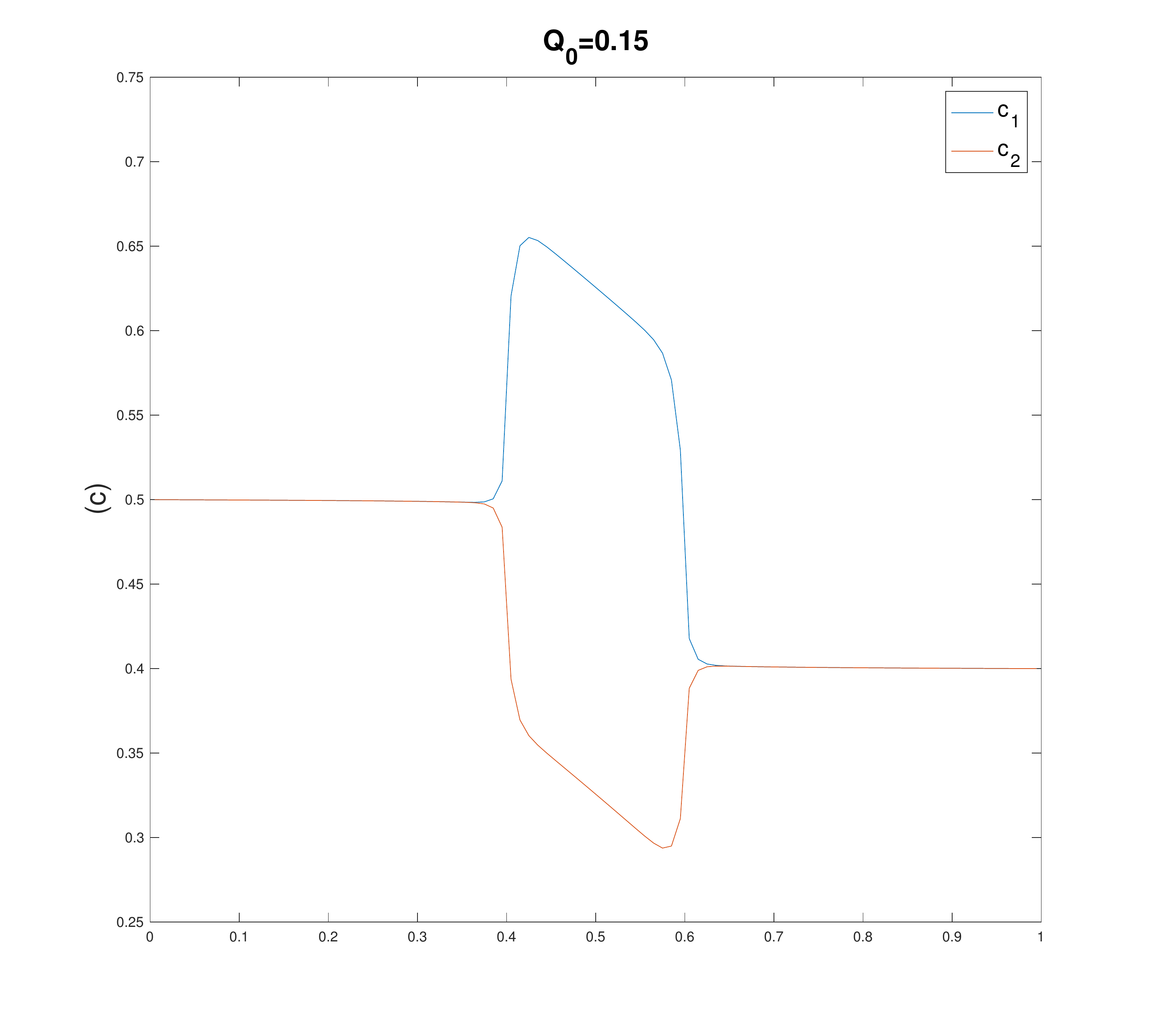}}\\
\subfigure{\includegraphics[width=0.30\linewidth]{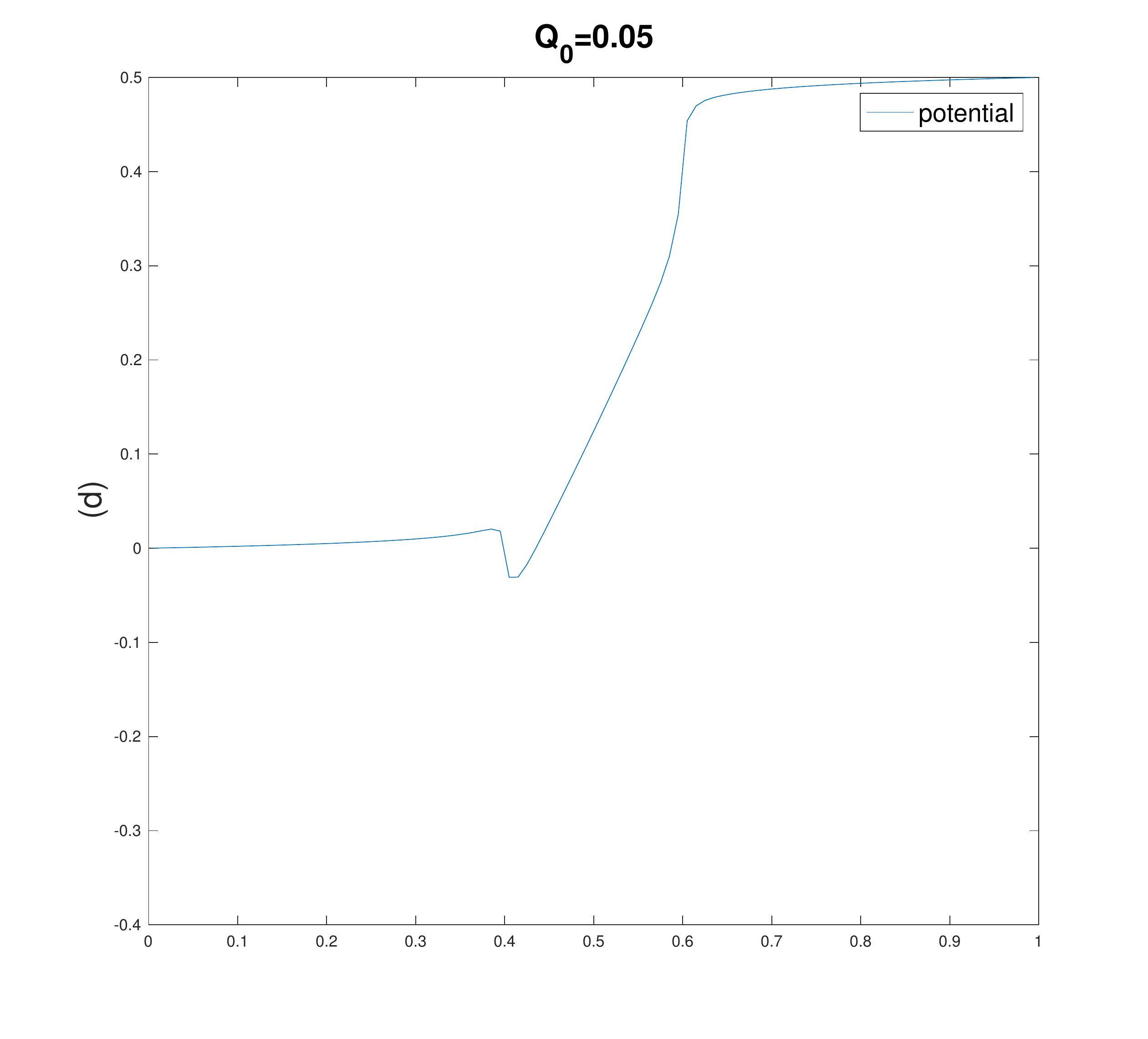}}
\subfigure{\includegraphics[width=0.30\linewidth]{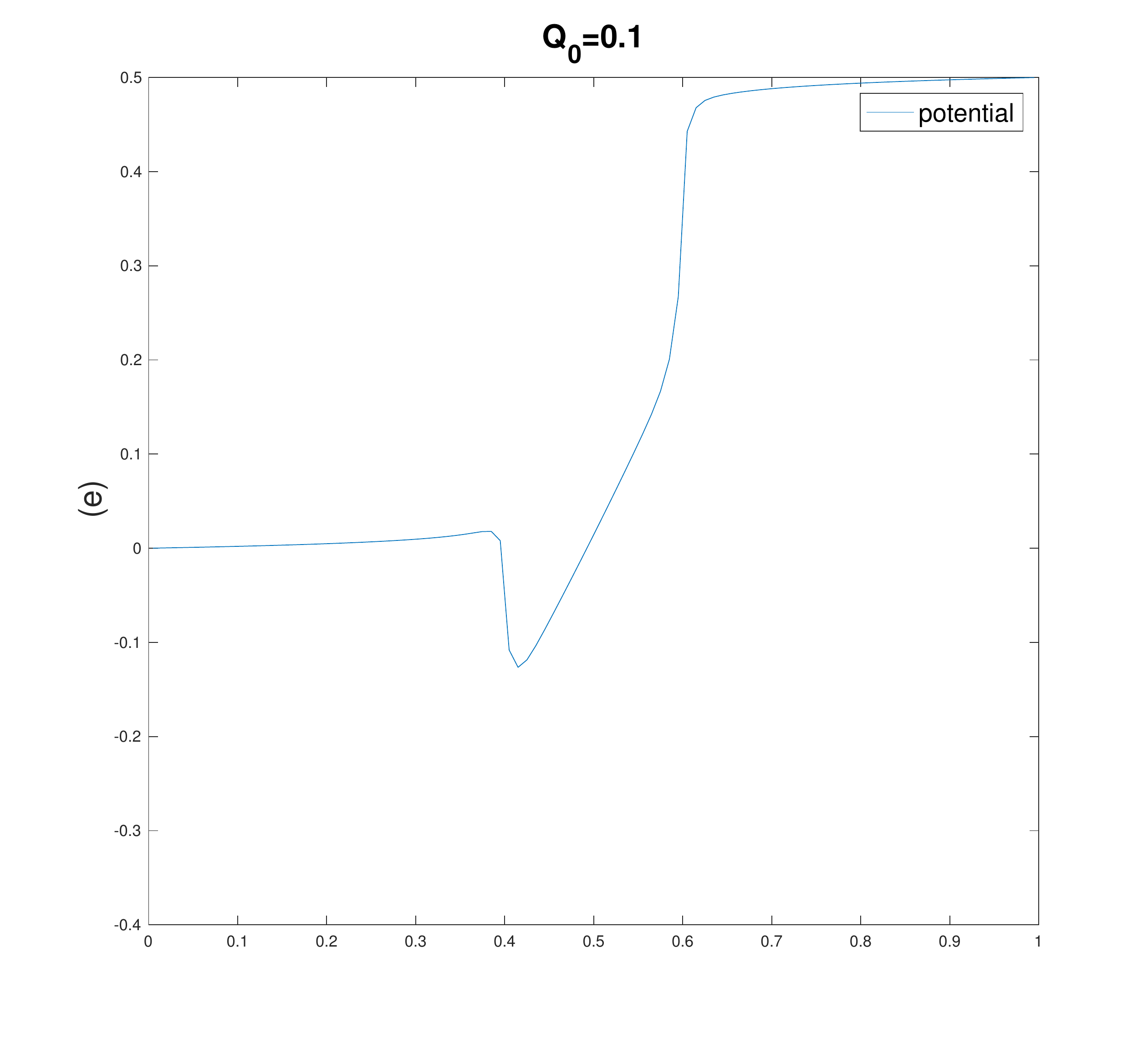}}
\subfigure{\includegraphics[width=0.30\linewidth]{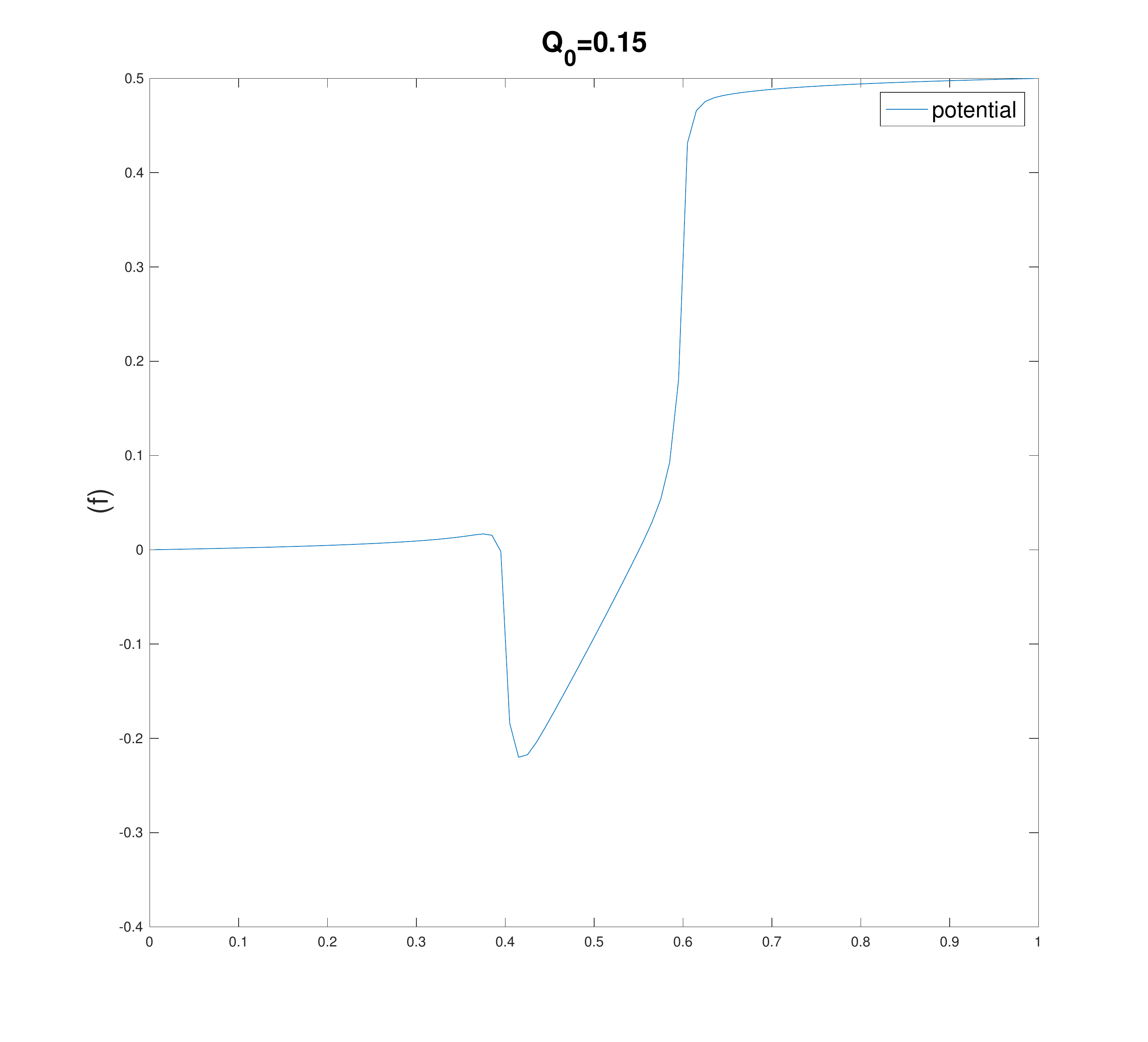}}
\end{figure}

\begin{figure}[!tbp]
\caption{I-V relation: (a) voltage for $V=0.5,\ 1,\ 3,\ 5,$ with $l_c=r_c=\frac{1}{5}$ and $Q_0=0.1$ at time $t=0.2,$ (b) current voltage relation.}
\centering  
\subfigure{\includegraphics[width=0.45\linewidth]{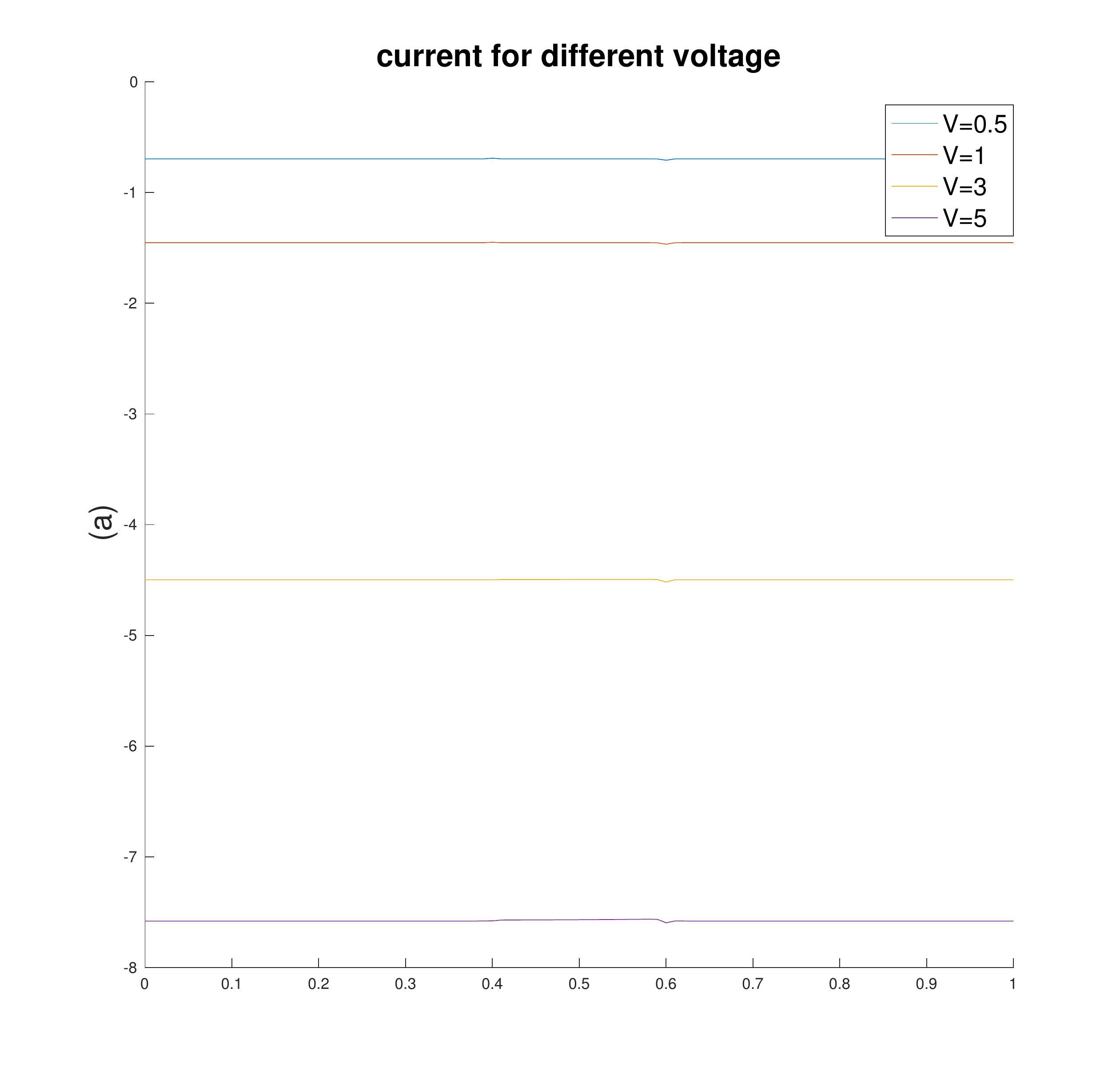}}
\subfigure{\includegraphics[width=0.45\linewidth]{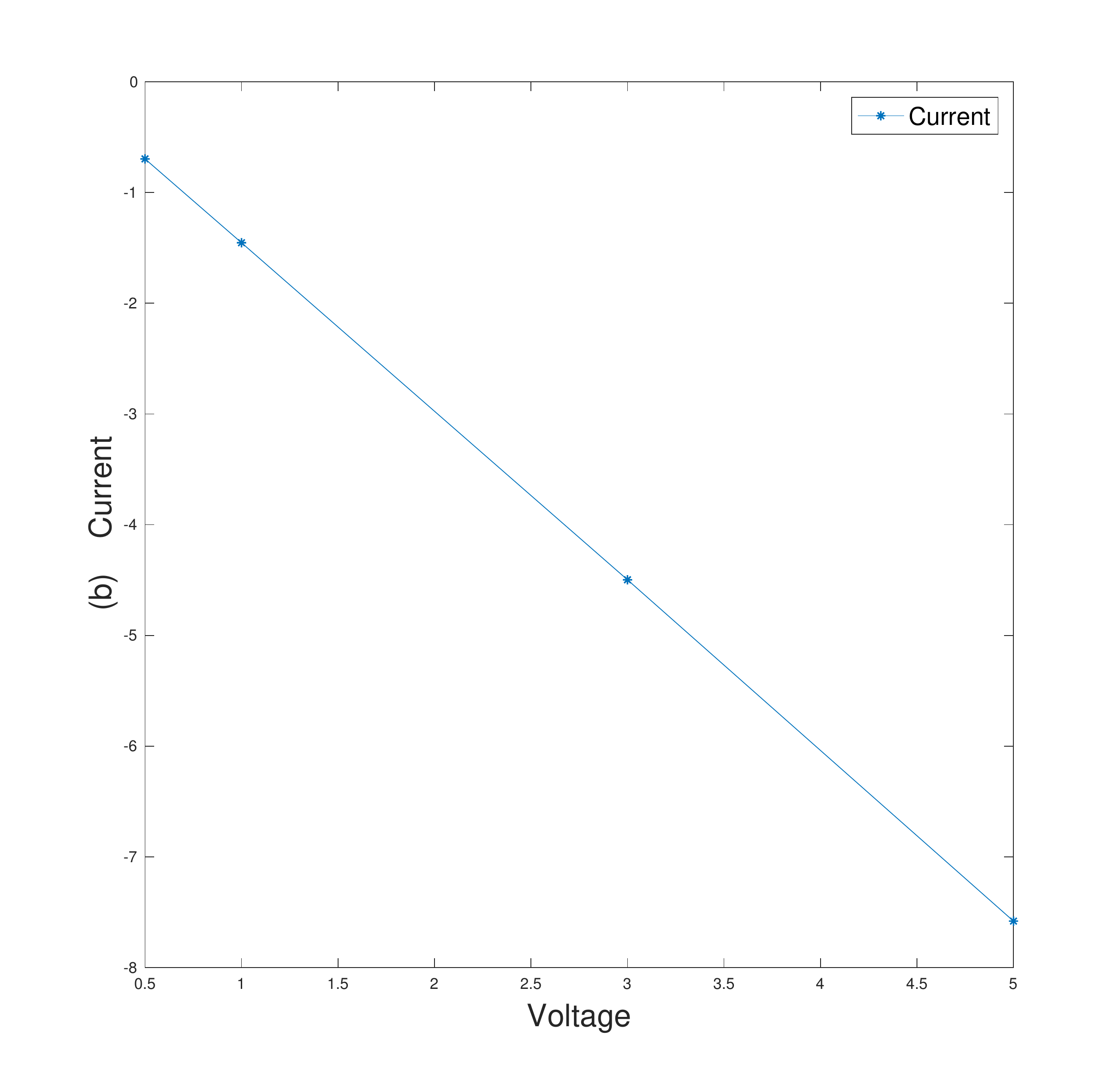}}
\end{figure}
\end{example}

\begin{example}\label{ex3} (No permanent charge in the channel $\rho=0$)  We still use problem with (\ref{PNPQ}), (\ref{IB}), (\ref{2Ax}) and (\ref{2Qx}) to test  the effects of the channel geometry, by  taking $\rho=0$, $r_f=20$ and varying $l_c$ and $r_c$. In the simulation we take $h=0.01,$ $\tau=5\times 10^{-5}$.  
 \begin{table}[ht]\label{ex13}
        \centering
                 \caption{\tiny{Times needed for reaching each steady state on Example \ref{ex3} when $\rho=0, r_f=20$ with different channel geometry}}
        \begin{tabular}{|c| c |c |c| c |c|c| c |}
            \hline
              channel parameters &   $||\psi^{n}-\psi^{n-1}||_{\infty}$ &  time $t_s$  &  iterations $n=t_s/\tau$  &  CPU time (sec) \\ [0.5ex] 
            \hline
            $r_c=l_c=\frac{1}{3}$ &   9.9876E-07&    0.0589&    1178  & 0.4647 \\
            \hline
          $r_c=l_c=\frac{1}{5}$  & 9.9928E-07& 0.0747&  1494 & 0.5529\\
          \hline
            $r_c=l_c=\frac{1}{11}$ & 9.9873E-07 & 0.0888 &  1776 &  0.6116
             \\ [1ex]
            \hline
        \end{tabular}                  
     \end{table} 
     
From Table 4 we see that $t_s=0.0888$ is the longest time needed for reaching the steady state, so simulation runs up to $t=0.1$. 

In Figure 5 are snap shots of solutions for different channel geometry. In the case of no permanent charge, there does not seem to be any layering phenomenon on $c_1$ and $c_2$:  $c_1$ and $c_2$ are rather close both inside the channel (linear) and inside the bath (constant). The profile for $\psi$ is quite similar. This is consistent with the analysis in \cite{WS15}, in which the authors showed that the density in the channel gets steeper as the channel gets narrower. 
We refer to \cite{WSzeroQ15} for a study of steady state solutions to (\ref{PNPQ}) in the case of $\rho=0.$  They proved that for $\epsilon >0$ small, there is a unique nonnegative steady state to problem (\ref{PNPQ}) and (\ref{IB}). 
\begin{figure}[!tbp]
\caption{Effects of channel geometry on steady state densities and potential with $\rho=0$: (a) densities at $t=0.1,$ for  $l_c=r_c=\frac{1}{3},\ \frac{1}{5}$ and $\frac{1}{11}$,  (b) potential profiles at $t=0.1$ for  $l_c=r_c=\frac{1}{3},\ \frac{1}{5}$ and $\frac{1}{11}$. }
\centering  
\subfigure{\includegraphics[width=0.45\linewidth]{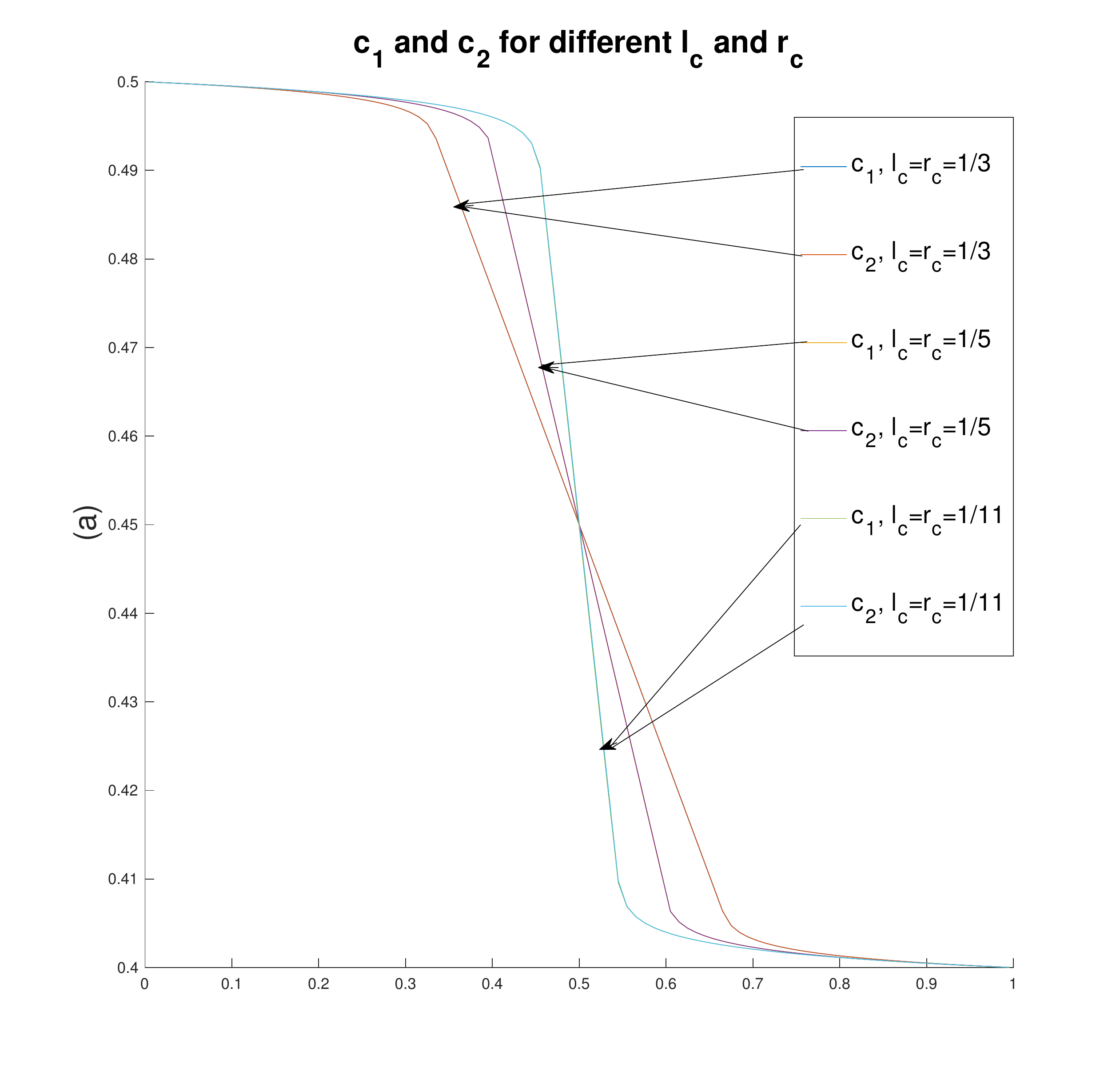}}
\subfigure{\includegraphics[width=0.45\linewidth]{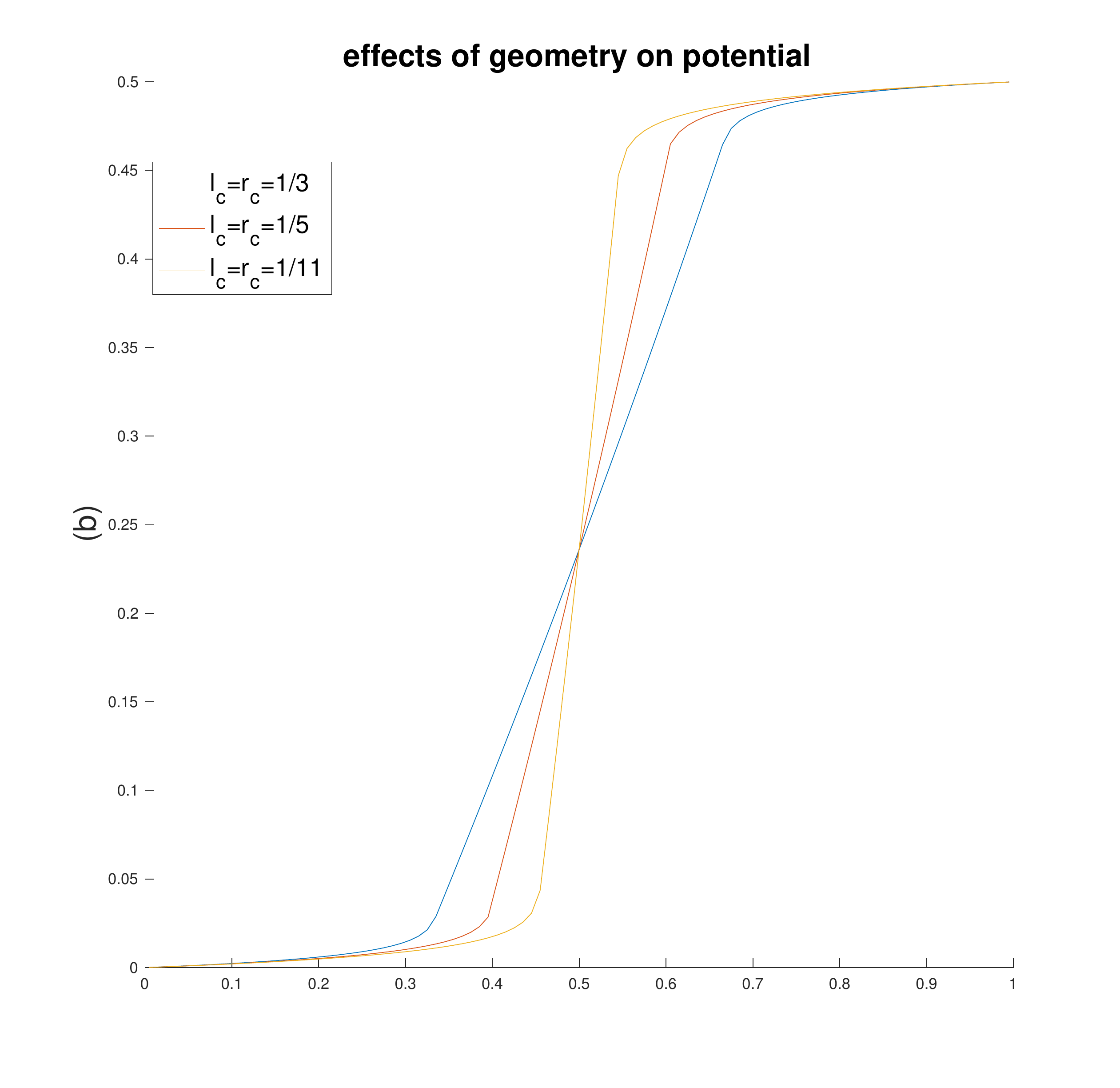}}
\end{figure}
\end{example}

 \begin{example}\label{ex4} (Variable diffusion coefficient and quadratic area function) We consider the system 
\begin{equation}\label{PNP4}
 \begin{aligned}
   A(x)  \partial_t c_1 &= \partial_x (A(x)D_1(x)( \partial_x c_1 +2c_1 \partial_x \psi), \quad x \in [-10, 10],\ t>0\\
    A(x)  \partial_t c_2 &= \partial_x (A(x)D_2(x)( \partial_x c_2 -3c_2 \partial_x \psi),  \quad x \in [-10, 10],\ t>0\\
    A(x)  \partial_t c_3 &= \partial_x (A(x)D_3(x)( \partial_x c_3 +c_3 \partial_x \psi),  \quad x \in [-10, 10],\ t>0\\
   -\frac{1}{A(x)}\partial_x (\epsilon A(x) \partial_x \psi) & = 2c_1-3c_2+c_3 -\rho(x),  \quad x \in [-10, 10],\ t>0, 
\end{aligned}
\end{equation}
where $\epsilon= 0.1,$ subject to boundary conditions
\begin{equation}\label{IB4}
\begin{aligned}
   c_i(\pm 10,t) &=0.5, \quad \psi(\pm 10,t)=0,\quad    t>0.
\end{aligned}
\end{equation}
As in \cite{GLE2018}, we choose  $A(x)=1+x^2,$ $D_i(x)=20(1-0.9e^{-x^4})$ and $\rho=Ce^{-x^4}$. This corresponds to problem (\ref{PNP2}) with   $z_1=2$, $z_2=-3$, $z_3=1$, and $c_{i,l}=c_{i,r}=0.5$, $V=0$. In this numerical test we take $h=0.1$, $\tau=10^{-3}$.

We take two different sets of initial data, first set  is  given as
\begin{equation}\label{c1}
\begin{aligned}
c^{in}_1(x)&= 0.5-0.5e^{-(x+4)^4},\\
c^{in}_2(x)&= 0.5+2e^{{-x}^4},\\
c^{in}_3(x)&= 0.5+e^{-(x-4)^4}.
\end{aligned}
\end{equation}
For the second set of initial data we take uniformly distributed random initial data $c^{0}_{i,j}\in (0, 1)$.  From Table 5 we see that $t_s=2.7410$ is the longest time needed for reaching the steady state, so simulation runs up to $t=3$. We vary the parameter $C$ to observe 
effects of the permanent charge. 

    \begin{table}[ht]\label{ex14}
        \centering
                 \caption{\tiny{Time needed for reaching steady state on Example \ref{ex4} when $C=1$ with different initial data}}
        \begin{tabular}{|c| c |c |c| c |c|c| c |}
            \hline
               initial data &   $||\psi^{n}-\psi^{n-1}||_{\infty}$ &  time $t_s$  &  iterations $n=t_s/\tau$  &  CPU time (sec) \\ [0.5ex] 
            \hline
          data (\ref{c1})   &  9.9999E-08&     2.7410&    2741 & 0.3874\\
            \hline
           random data & 9.9864E-08 & 2.0440 &  2044& 0.3167
             \\ [1ex]
            \hline
        \end{tabular}                  
     \end{table}

In Figure 6 (top three) are snap shots of solutions for initial data (\ref{c1}).  Varying $C$, we can see that $\max c_1-\min c_2$ (or $\max c_3-\min c_2$) increases in terms of $C$.

In Figure 6 (bottom three) are snap shots of solutions for random initial data, we see that the choice of initial data does not affect the steady state densities. 
  
  \begin{figure}[!tbp]
\caption{Effects of permanent charge and initial data on steady state densities: (a) is initial data profile (4.8), (b)-(c) are density profile at $t=3$ for $C=1$ and $C=2$ respectively, (d) is random initial data profile, (e)-(f) are density profile for $C=1$ and $C=2$.}
\centering  
\subfigure{\includegraphics[width=0.30\linewidth]{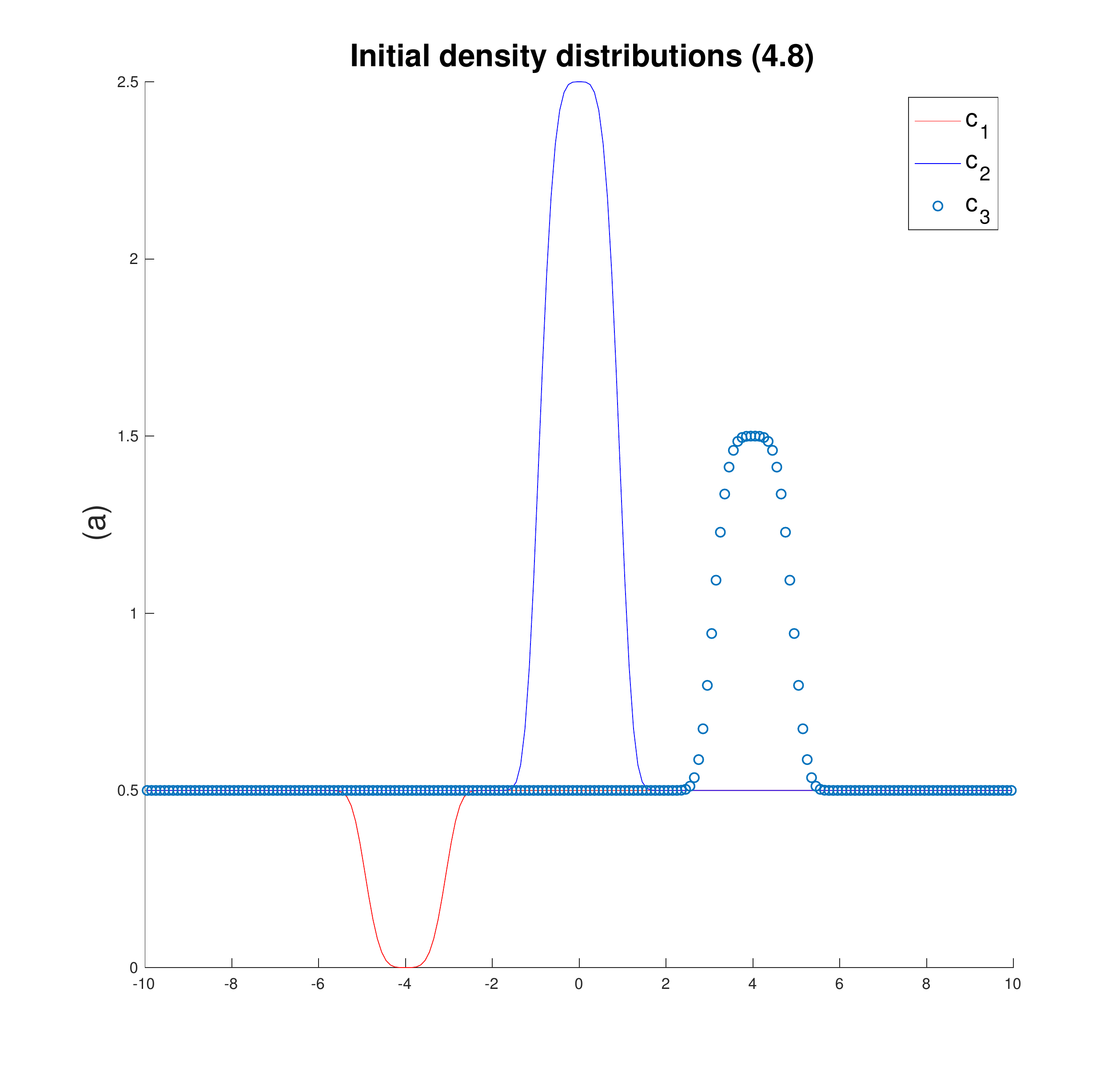}}
\subfigure{\includegraphics[width=0.30\linewidth]{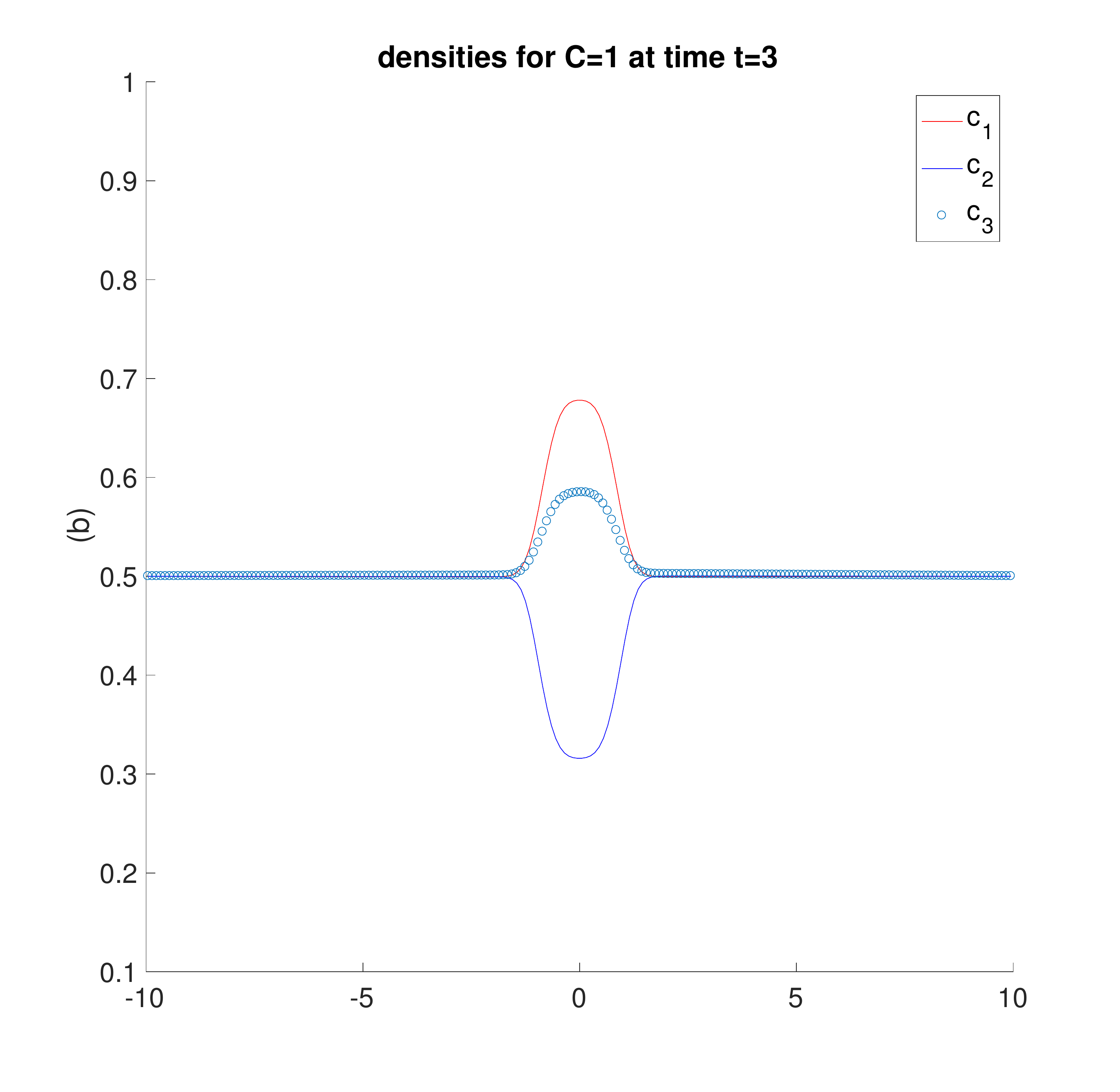}}
\subfigure{\includegraphics[width=0.30\linewidth]{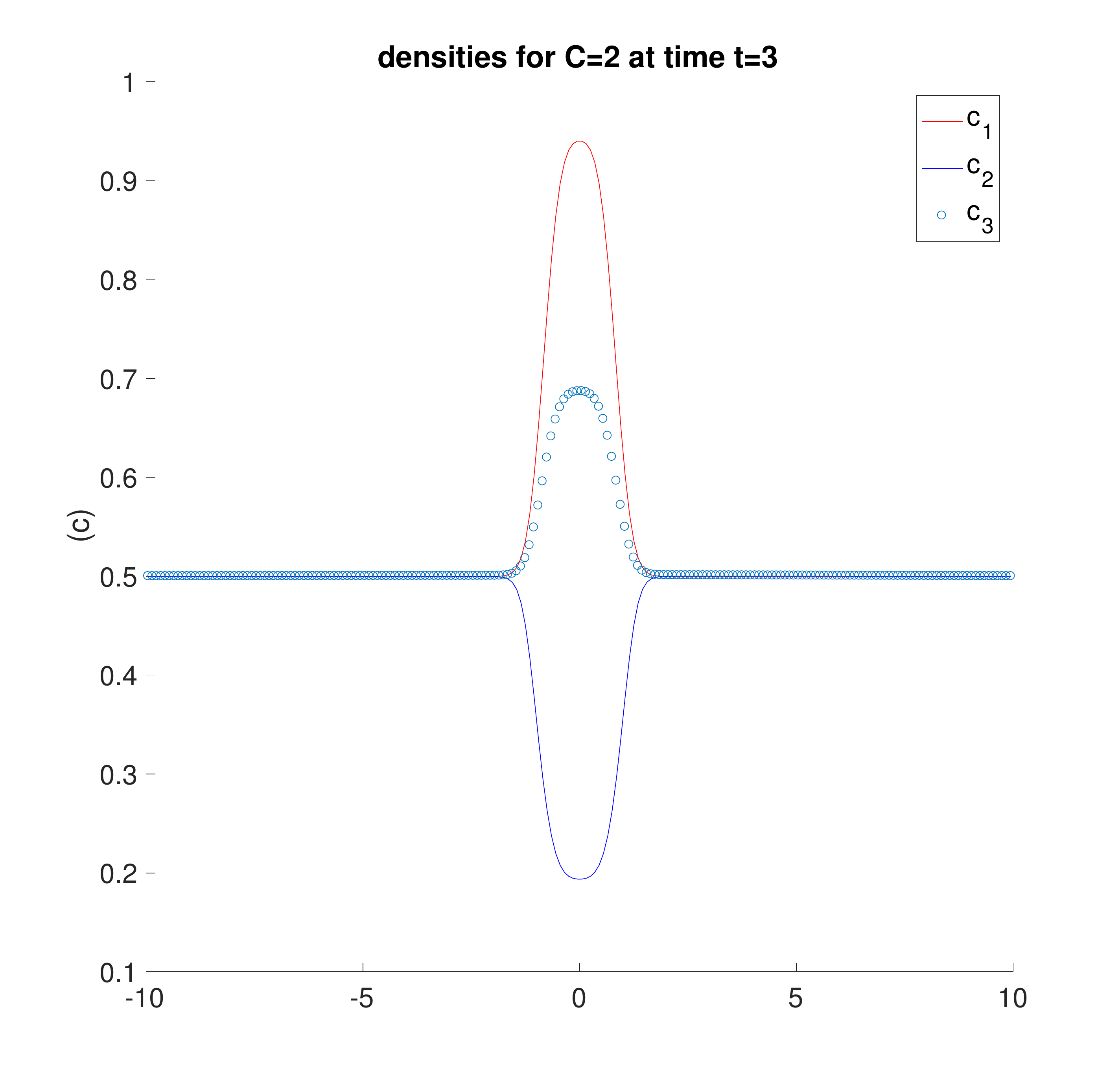}}\\
\subfigure{\includegraphics[width=0.30\linewidth]{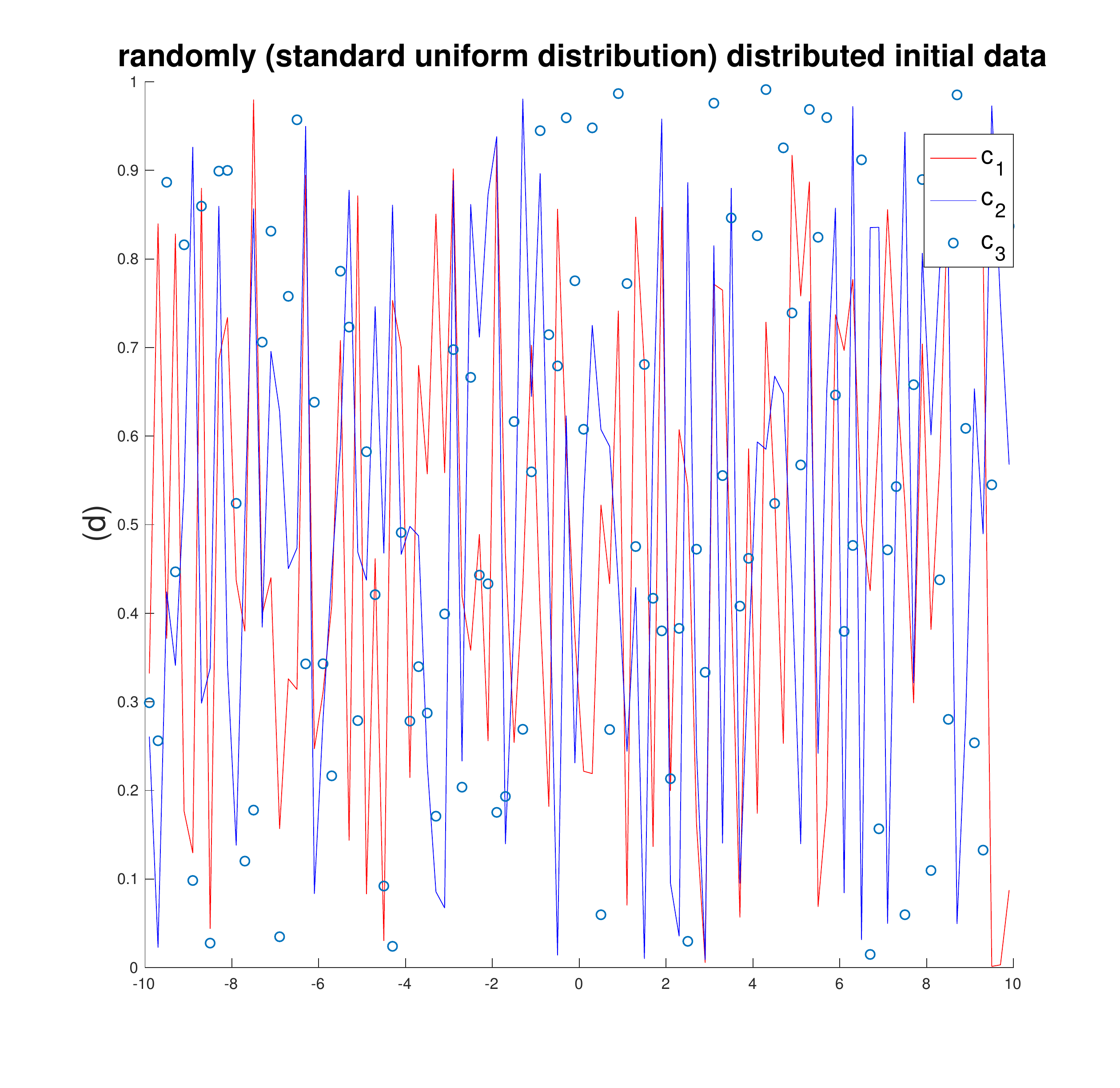}}
\subfigure{\includegraphics[width=0.30\linewidth]{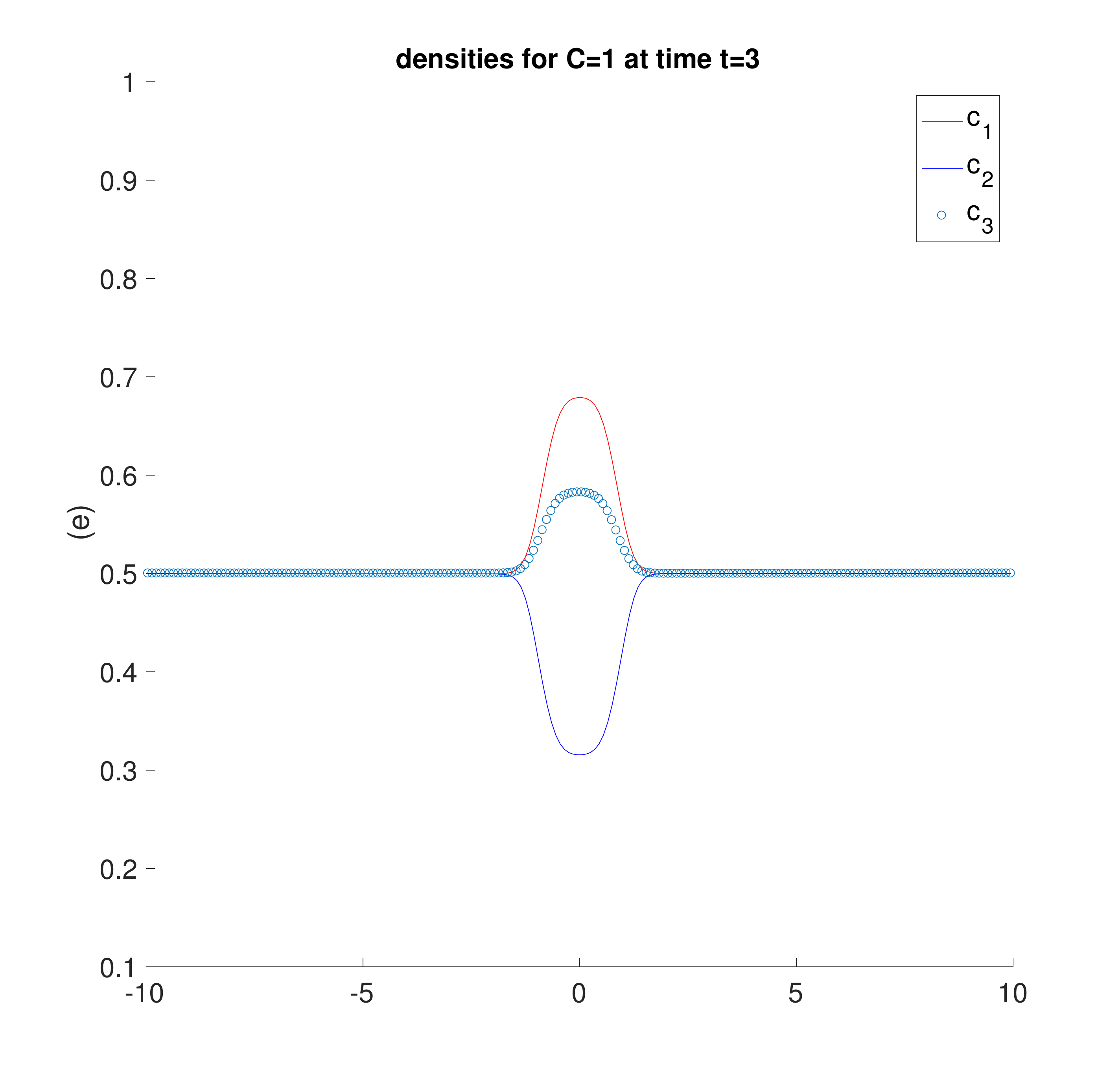}}
\subfigure{\includegraphics[width=0.30\linewidth]{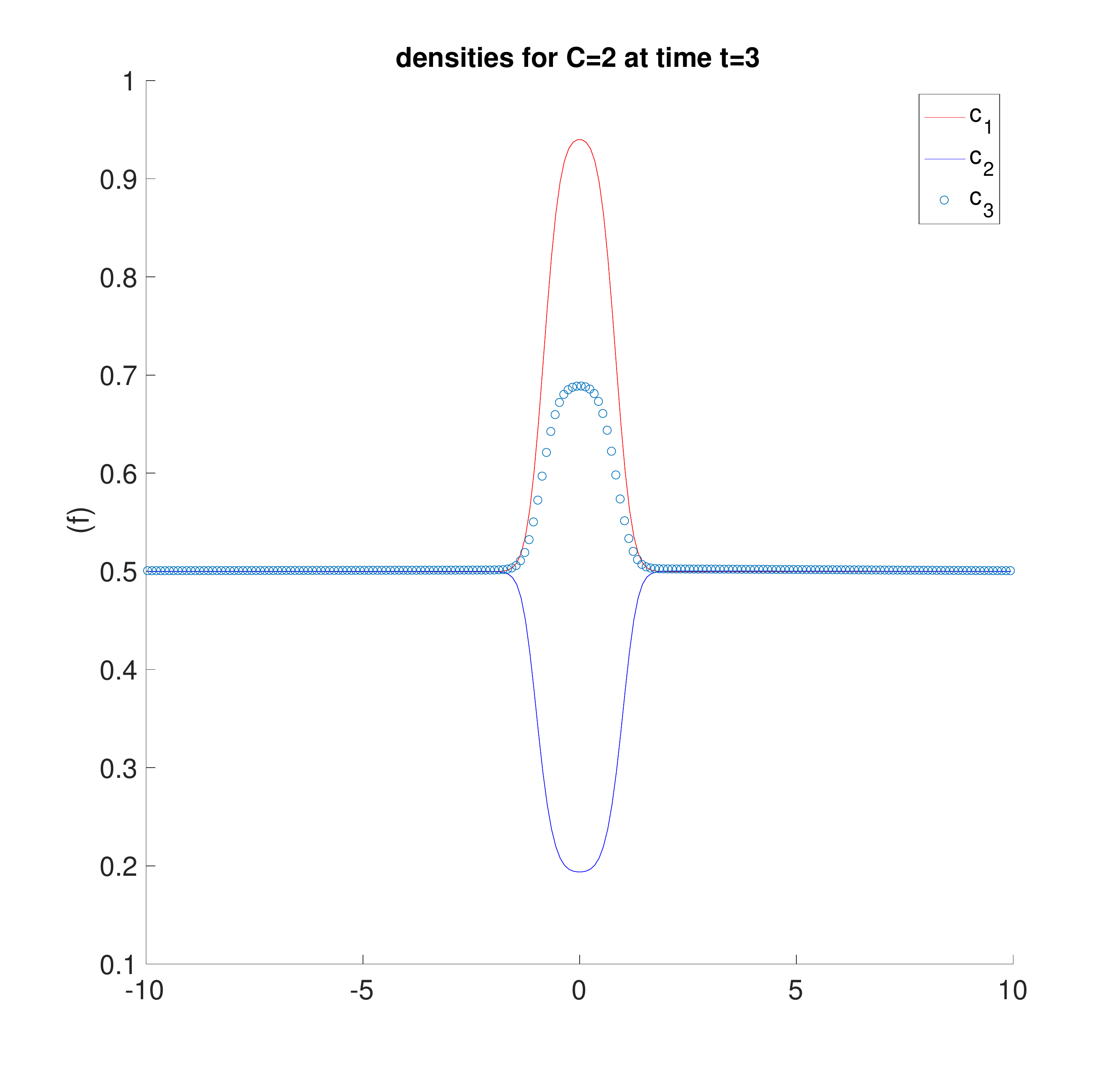}}
\end{figure}
\end{example}

 \subsection{Mass conservation and free energy dissipation}  In this numerical test we demonstrate mass conservation and free energy dissipation properties.
 
\begin{example}{(Zero flux + Robin boundary conditions)} In this example we consider (\ref{PNP4}) with initial condition (\ref{c1}) and boundary condition
\begin{equation}\label{Zf5}
\begin{aligned}
&\partial_x c_i+z_ic_i\partial_x \psi =0,\quad x=-10,10, \quad t>0,\\
& (- \eta \partial_x \psi+\psi)|_{x=-10}=-0.1,\quad (\eta \partial_x \psi+\psi)|_{x=10}=0.1, \quad t>0.
\end{aligned}
\end{equation}
We choose same $A(x), D_i(x), \rho(x)$ and $z_i$ as in Example 4.4 and choose $\eta=0.1$, $\epsilon=0.1$. In this numerical test we use scheme (\ref{fullyCi})-(\ref{Ps2}) and (\ref{ZfBC}), with $h=0.1$, $\tau=10^{-3}$.

 \begin{figure}[!tbp]
\caption{Energy dissipation and mass conservation: (a) density profiles at $t=15$ for $C=1$, (b) energy dissipation and mass conservation.}
\centering  
\subfigure{\includegraphics[width=0.45\linewidth]{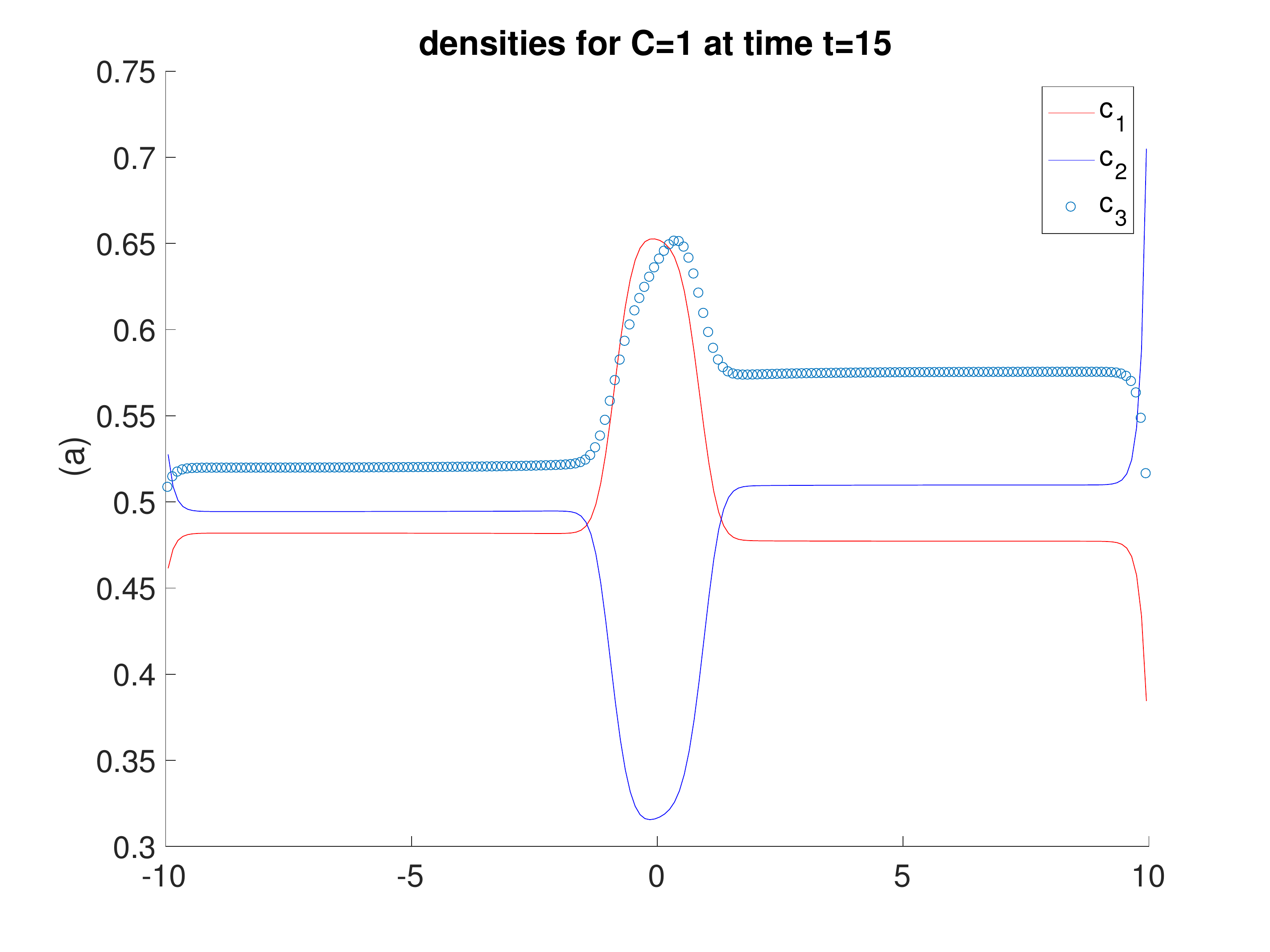}}
\subfigure{\includegraphics[width=0.45\linewidth]{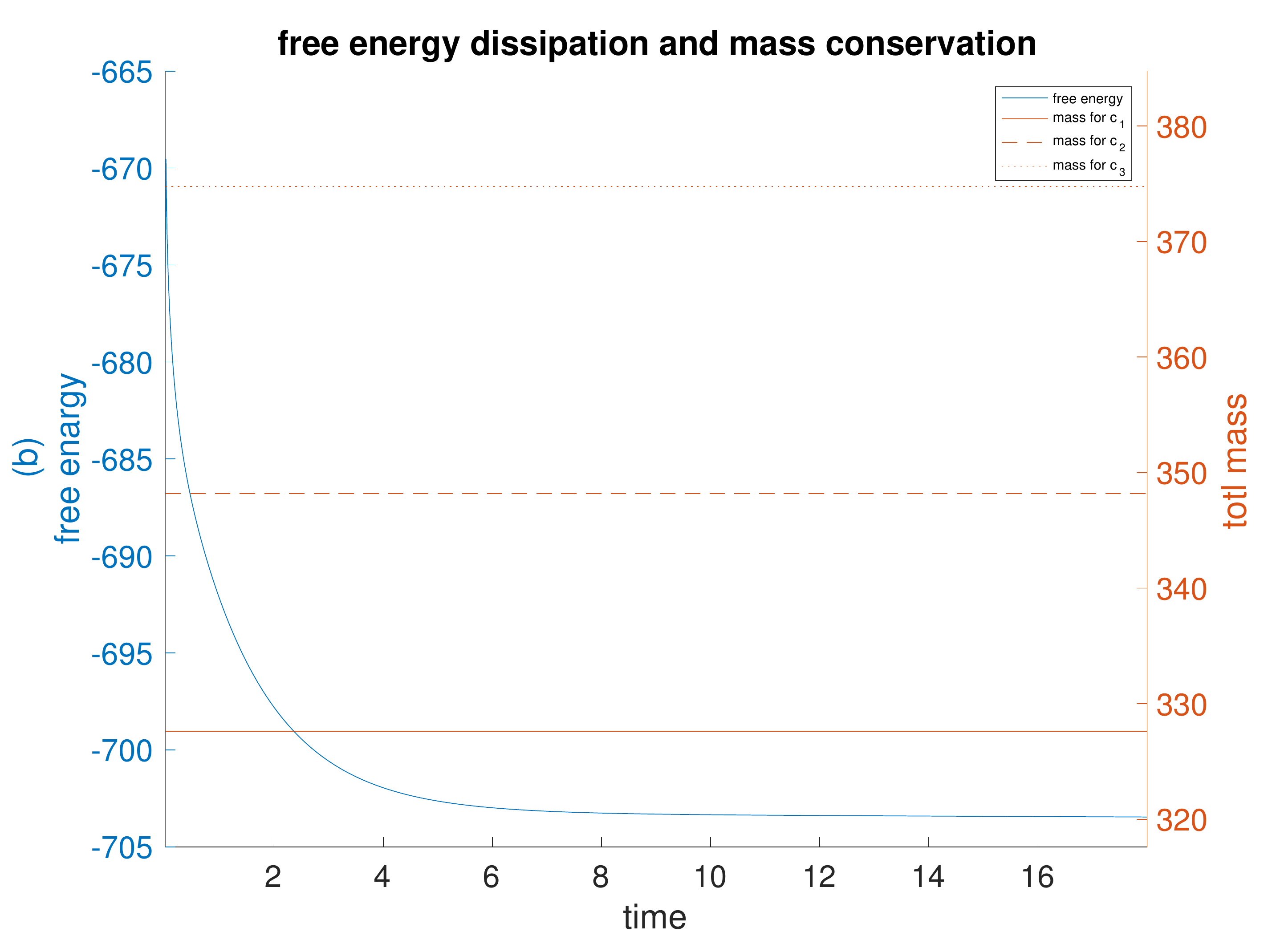}}
\end{figure}
In Figure 7 (left) are snap shots of solutions for initial data (\ref{c1}), (right) is free energy for the system and total mass for each species, which confirms energy dissipation and mass conservation properties as proved in Theorem 3.2.
\end{example}
\section{Concluding Remarks}
In this paper, we have developed an unconditional positivity-preserving finite-volume method for solving initial boundary value problems for the reduced Poisson-Nernst-Planck system. Such a reduced system has been used as a good approximation to the 3D ion channel problem. By writing the underling system in non-logarithmic Landau form and using a semi-implicit time discretization, we constructed a simple, easy-to-implement numerical scheme which proved to satisfy positivity independent of time steps and the choice of Poisson solvers. Our scheme also preserves total mass and satisfies a free energy dissipation property for zero flux boundary conditions. Extensive numerical tests have been presented to simulate ionic channels in different settings.

\section*{Acknowledgments} The authors would like to thank Robert Eisenberg for stimulating discussions on PNP systems and their role in 
modeling ion channels.  This research was supported by the National Science Foundation under Grant DMS1812666 and by NSF Grant RNMS (KI-Net) 1107291.

\end{document}